\documentclass[unnumsec,webpdf,contemporary,large]{oup-authoring-template}

\theoremstyle{thmstyleone}
\newtheorem{theorem}{Theorem}
\newtheorem{proposition}[theorem]{Proposition}
\newtheorem{lemma}{Lemma}
\theoremstyle{thmstyletwo}

\theoremstyle{thmstylethree}
\newtheorem{definition}{Definition}
\newtheorem{corollary}{Corollary}
\newtheorem{problem}{Problem}
\newtheorem{notation}{Notation}
\newtheorem{assumption}{Assumption}

\newcommand{\cal}{\mathcal}
\usepackage{amsmath}
\usepackage{bm}
\usepackage{doi}
\usepackage{url}

\begin{document}

\journaltitle{PNAS Nexus}
\DOI{DOI HERE}
\copyrightyear{2022}
\pubyear{2022}
\access{Advance Access Publication Date: Day Month Year}
\appnotes{Manuscript}

\firstpage{1}

\title{Dimension reduction of dynamics on modular and heterogeneous directed networks}

\author[a,b,$\ast$]{Marina Vegu\'{e}}
\author[a,b]{Vincent Thibeault}
\author[a,b,c]{Patrick Desrosiers}
\author[a,b]{Antoine Allard}

\authormark{Vegu\'{e} et al.}

\address[a]{\orgdiv{D\'{e}partement de physique, de g\'{e}nie physique et d'optique}, \orgname{Universit\'{e} Laval}, \orgaddress{\street{2325 rue de l'Universit\'{e}}, \postcode{G1V 0A6}, \state{Qu\'{e}bec}, \country{Canada}}}
\address[b]{\orgdiv{Centre interdisciplinaire en mod\'{e}lisation math\'{e}matique}, \orgname{Universit\'{e} Laval}, \orgaddress{\street{2325 rue de l'Universit\'{e}}, \postcode{G1V 0A6}, \state{Qu\'{e}bec}, \country{Canada}}}
\address[c]{\orgdiv{CERVO Brain Research Center}, \orgaddress{\street{2301 avenue d'Estimauville}, \postcode{G1E 1T2}, \state{Qu\'{e}bec}, \country{Canada}}}

\corresp[$\ast$]{
\href{email:mavel15@ulaval.ca}{mavel15@ulaval.ca}
}

\abstract{
Dimension reduction is a common strategy to study non-linear dynamical systems composed by a large number of variables. The goal is to find a smaller version of the system whose time evolution is easier to predict while preserving some of the key dynamical features of the original system. Finding such a reduced representation for complex systems is, however, a difficult task. We address this problem for dynamics on weighted directed networks, with special emphasis on modular and heterogeneous networks. We propose a two-step dimension-reduction method that takes into account the properties of the adjacency matrix. First, units are partitioned into groups of similar connectivity profiles. Each group is associated to an observable that is a weighted average of the nodes' activities within the group. Second, we derive a set of conditions that must be fulfilled for these observables to properly represent the original system's behavior, together with a method for approximately solving them. The result is a reduced adjacency matrix and an approximate system of ODEs for the observables' evolution. We show that the reduced system can be used to predict some characteristic features of the complete dynamics for different types of connectivity structures, both synthetic and derived from real data, including neuronal, ecological, and social networks.  Our formalism opens a way to a systematic comparison of the effect of various structural properties on the overall network dynamics. It can thus help to identify the main structural driving forces guiding the evolution of dynamical processes on networks.
}

\keywords{dimension reduction, networks, non-linear dynamics, community structure, spectral decomposition}

\maketitle

\section{Introduction}

Dimension-reduction methods seek to find a low-dimensional representation of a high-dimensional system. This representation should preserve various key properties of the original system while being more amenable to analysis in order to provide insights on the inner workings and the long-term behavior of the system. 
The dimension of a proper reduction can also inform on the effective dimension of the original system, namely the extent to which it can be compressed into a simpler form. These reasons make dimension reduction not only necessary for practical or computational purposes but also interesting from a purely theoretical standpoint.

Dimension-reduction methods appear in many areas of science and under different names. For example, the ubiquity of high-dimensional data and the difficulty of extracting the relevant patterns have made dimension reduction an essential problem in data analysis~\cite{Donoho2000, Aggarwal2014, Brunton2019}. Dimension reduction is also relevant in the context of linear control systems, where it is referred to as \emph{model order reduction}~\cite{Antoulas2005, Schilders2008}. 
Methods to reduce the dimension of systems of ordinary differential equations (ODEs) have a relatively long tradition in chemistry, where they are often known as \emph{lumping}~\cite{Wei1969, Kuo1969, Li1990}. Yet, these latter works focus on the kinetic theory of molecular reactions and the methods are often applied to dynamics involving only a few number of molecules.

The complexity inherent to large, non-linear dynamical systems of interacting units (ecological communities, neuronal assemblies, etc.) makes them very difficult to study~\cite{mitchell2009complexity, thurner2018introduction, ladyman2020complex}. One of the big challenges of network science is to find ways to approximate these systems by ones of reduced dimension so as to make them be more tractable both analytically and computationally. 
How to construct a reduced version of a generic complex system is, however, still an open question~\cite{cheng2021model}. Dimension reduction is especially challenging when the original system is highly heterogeneous, i.e., when the rules that define the units' dynamics and their interactions vary largely across the different units.  A standard way to model the influence that each unit (or node) has on the activity of the others is by means of a weighted adjacency matrix plus a set of coupling functions of the nodes' activities~\cite{boccaletti2006complex, barzel2013universality}. The adjacency matrix can be either a constant of the system or change in time depending on the nodes' activity, as in plastic neuronal networks~\cite{gerstner2002mathematical, zenke2017hebbian}. According to this framework, a heterogeneous system is a system in which either the adjacency matrix or the functions that define the dynamics exhibit a large variability across the different nodes.

Several dimension-reduction strategies for dynamics on networks have been proposed in recent years. Gao~et~al.~\cite{gao_universal_2016} developed a method for reducing an $N$-dimensional system to a 1-dimensional one that approximately models the temporal evolution of an effective activity variable. This variable is an average of the nodes' activities, weighted by their outgoing degrees. This approach was used to predict the resilience of different real-world networks under several types of perturbations. Jiang~et~al.~\cite{jiang_predicting_2018} defined a strategy for predicting the resilience and tipping points of bipartite plant-pollinator networks in which the reduced system is 2-dimensional.  The two variables in the reduced dynamics then correspond to the (weighted) average abundance of the plants and pollinators, respectively. Recently, Tu~et~al.~\cite{tu_dimensionality_2021} proposed a dimension reduction framework designed for systems that are heterogeneous in terms of the functions that define the self-dynamics and the coupling-dynamics of the nodes. The reduced dynamics in this case is 1-dimensional and incorporates a variable number of control parameters.

Although 1- and 2-dimensional reductions have been proven to be effective in some cases, the large heterogeneity of some real-world networks \cite{broido2019scale} and their tendency to form intricate community structures \cite{girvan2002community, sporns_modular_2016} make them unlikely to be well understood by means of systems of such a low dimension. References~\cite{laurence_spectral_2019, thibeault_threefold_2020} introduced formalisms for reducing an $N$-dimensional system to a system of arbitrary dimension $n < N$. The $n$ variables of the reduced system are constructed by taking into account the spectral properties of the adjacency matrix (as well as the matrix of self-coupling dynamical parameters in Ref.~\cite{thibeault_threefold_2020}), regardless of the specific form of the coupling functions. The reduction method defined in Ref.~\cite{laurence_spectral_2019} is appropriate for undirected networks with weakly coupled communities or with a bipartite structure.  Reference~\cite{thibeault_threefold_2020} generalized the formalism to dynamical systems with heterogeneous dynamical parameters on generic undirected networks, including the ones with strongly coupled modules. Both approaches have been tested on undirected networks with homogeneous group connectivity, which is also a very common structural assumption to study synchronization of coupled oscillators through dimension reduction~\cite{Ott2008a, Pikovsky2008, Gfeller2008, Hancock2018a, Snyder2020a}. It remained unclear, however, how these approaches could be adapted to reduce the dimension of dynamics on directed networks with significant heterogeneous connectivity.

We develop a dimension-reduction method for dynamical systems on directed networks that are possibly modular but whose modules are significantly heterogeneous. We thus discard all the structural limitations that were previously imposed on networks. The method is defined by a two-step process illustrated in Fig.~\ref{fig spectral red}. First, the nodes are classified into $n$ groups of nodes that share similar connectivity properties (these could correspond to the modules, or to subgroups within the modules, in the case of modular networks). The variables of the reduced system, the \emph{observables}, are weighted averages of the node activities within each group. This means that the number of groups in the node partition, $n$, corresponds to the dimension of the reduced system. We refer to the vectors that specify these group averages as the \emph{reduction vectors}. Generically, the observables' dynamics cannot be expressed as a closed system of ODEs. The second step consists of approximately closing it. For this, we approximate the dynamics' functions and find conditions on the reduction vectors that make this approximation as accurate as possible. These conditions, that we dub the \emph{compatibility equations}, determine the reduction vectors and the parameters of the reduced system, which define a \emph{reduced adjacency matrix}. The result is a \emph{reduced system} of dimension $n$ on a \emph{reduced network}.  

The paper is structured as follows. First, we present the dimension reduction strategy together with a method for approximately solving the compatibility equations. Then, we test the reductions in different systems, for various types of dynamics and synthetic network topologies. 
We finally apply our method to systems inspired in three types of real-world networks. 
All the theoretical results used to derive the equations presented in the paper are detailed in the Supplementary Information (SI).

\section{Dimension reduction strategy}\label{sec1}

\subsection{Overview}
In many real-world networks, the nodes can be partitioned into groups of densely connected nodes (commonly referred to as \emph{modules})~\cite{hartwell_molecular_1999, newman_modularity_2006, ravasz_hierarchical_2002, sporns_modular_2016}. Usually, a module corresponds to a group of nodes that are also involved in a common task, which suggests that the structural modularity seen in these networks reflects a modularity at the functional level~\cite{ravasz_hierarchical_2002, sporns_modular_2016}. Other networks have groups of sparsely connected nodes whose connectivity profiles with nodes in other groups are similar. For example, in plant-pollinator networks, pollinators interact with plants and plants with pollinators, but it is assumed that plants do not directly interact with other plant species and analogously for pollinators (i.e. the interaction network is bipartite). We can thus say that the set of plants and the set of pollinators constitute two different groups in terms of their connectivity profiles. Either if the structure of these networks tends to be assortative (nodes of similar connectivity properties tend to be connected) or dissortative (nodes of similar connectivity properties tend to be disconnected), they have in common the fact that their nodes can be classified into groups of nodes that share similar connectivity properties. Throughout this paper, we use the terms \emph{modularity} and \emph{community structure} to refer to this network property, regardless of whether the groups are densely or sparsely connected, and \emph{modules} or \emph{communities} to refer to these groups.

We present a strategy to reduce the dimension of a dynamical system taking into account the community structure of the interaction network. 
We assume that the $N$ nodes of the original system interact by means of a 
not necessarily symmetric adjacency matrix $\bm{W} = (w_{ij})_{i,j}\in\mathbb{R}^{N\times N}$ and that their activities evolve in time according to a dynamics of the form
\begin{equation}
  \dot{x}_i = f(x_i) + \sum \limits_{j=1}^N w_{ij} \, g(x_i, x_j)\ ,
  \label{system}
\end{equation}
with $i \in \{1, \cdots, N\}$, and where $x_i$ is a real-valued function of a real variable (time) that represents the activity of node $i$. The functions $f$ and $g$ are generic functions of class $\mathscr{C}^1$ (i.e., 1-time differentiable with continuous derivatives). The function $f$ defines the self-dynamics of the nodes, while $g$ accounts for the dynamical coupling between pairs of nodes. The weight $w_{ij}\in \mathbb{R}$ encodes the strength of the interaction from node $j$ to node $i$.

This model assumes that node heterogeneity comes from the adjacency matrix itself (the functions $f$ and $g$ are the same for all the nodes). It is therefore reasonable to assume that nodes with a similar connectivity profile will have similar activities. We aim to construct a reduced version of the system by defining a set of \emph{linear observables}, each of them representing a weighted average of the node activities within each of the connectivity-based communities. Thus, our reduced system has the same dimension as the number of modules in the original network. 

Let us suppose that we know the community structure in our network. This means that we have a \emph{partition} of the nodes into groups: each node belongs to one and only one of these groups. We define $n$ as the number of groups and denote the groups by $G_1, \cdots, G_n$ and their corresponding observables by $\cal{X}_1, \cdots, \cal{X}_n$. For all $\nu \in \{1,\cdots,n\}$, the observable $\cal{X}_\nu$ is a linear combination of the activities of nodes in group $G_\nu$: $\cal{X}_\nu$ is defined by a non-negative, normalized weight vector $\bm{a}_\nu = (a_{\nu i})_{i=1}^N \in \mathbb{R}^N$:
\begin{equation}
\begin{array}{lllll}
  \cal{X}_\nu := \displaystyle \sum \limits_{i=1}^N a_{\nu i} x_i, 
  && \displaystyle \sum \limits_{i=1}^N a_{\nu i} = 1,
  && a_{\nu i} = 0 \text{ if } i \notin G_\nu. \\
\end{array}
\label{observables}
\end{equation}
We call $\bm{a}_1, \cdots, \bm{a}_n$ the \emph{reduction vectors}. 

To reduce the dimension of our system we need to specify how to map the original dynamics [Eq.~\eqref{system}] into a reduced dynamics for the $n$ observables. This implies (a) specifying what the vectors $\bm{a}_1, \cdots \bm{a}_n$ are, and (b) providing a system of ODEs for the temporal evolution of the observables. We provide the details in the following sections. A rigorous derivation of all the results is provided in the SI.

\subsection{Group adjacency and in-degree matrices}
Let us denote by $m_1, \cdots, m_n$ the sizes of the groups in our network. Without loss of generality we can suppose that the nodes have been reordered so that the indices of nodes in each group are consecutive numbers:
\begin{equation}
  G_\nu = \left\lbrace 1 + \sum \limits_{\rho=1}^{\nu-1} m_\rho, \cdots, \sum \limits_{\rho=1}^{\nu} m_\rho \right\rbrace.
\end{equation}
This allows us to express the adjacency matrix $\bm{W}$ in the block form
\begin{equation}
  \bm{W} = \left( 
    \begin{array}{ccc}
    \bm{W}_{11} & \cdots & \bm{W}_{1n} \\
    \vdots & \ddots & \vdots \\
    \bm{W}_{n1} & \cdots & \bm{W}_{nn}
  \end{array} \right),
  \label{decomposed adjacency matrix}
\end{equation}
where $\bm{W}_{\nu \rho}$ is the submatrix of size $m_\nu \times m_\rho$ that includes all the interaction weights from nodes in group $G_\rho$ to nodes in group $G_\nu$. In a similar way we can define the group-to-group weighted in-degree diagonal matrix $\bm{K}_{\nu \rho} = \text{diag} ( k_{i_1}^\rho, \cdots, k_{i_{m_\nu}}^\rho )$, where $\{i_1, \cdots, i_{m_\nu}\} = G_\nu$ and $k_i^\rho = \sum_{j \in G_\rho} w_{ij}$ is the weighted in-degree of node $i$ only taking into account connections that come from nodes in group $G_\rho$. Thus, the global diagonal in-degree matrix can be expressed as a function of the group-to-group in-degree matrices as
\begin{equation}
  \bm{K} = \left( 
  \begin{array}{ccc}
    \bm{K}_{11} + \cdots + \bm{K}_{1n} & \cdots & 0 \\
    \vdots & \ddots & \vdots \\
    0 & \cdots & \bm{K}_{n1} + \cdots + \bm{K}_{nn}
  \end{array} \right).
  \label{decomposed degree matrix}
\end{equation}

\subsection{Reduction vectors and approximate reduced dynamics}
From Eqs.~\eqref{system} and \eqref{observables}, the exact temporal evolution of a given observable $\cal{X}_\nu$ is
\begin{equation}
  \cal{ \dot{X}}_\nu 
    = \displaystyle \sum \limits_{i=1}^N a_{\nu i} f(x_i) + \sum \limits_{i,j=1}^N a_{\nu i} w_{ij} g(x_i,x_j).
  \label{Xnu dot}
\end{equation}
Our objective is to rewrite Eq.~\eqref{Xnu dot} in a closed form, that is, to make it be a function of the observables only. This cannot be fulfilled exactly in general, but we can work on Eq.~\eqref{Xnu dot} so that it admits an \emph{approximate} closed form. The general idea is the following: first, approximate $f(x_i$) and $g(x_i,x_j)$ by Taylor polynomials around the observables associated to the groups to which nodes $i$ and $j$ belong, and, second, introduce these approximations in Eq.~\eqref{Xnu dot} to find conditions on the set of reduction vectors $\{\bm{a}_\nu\}_\nu$ which ensure that the resulting equations admit an approximate closed form. The justification behind using Taylor approximations is the following: if the node partition reflects a true organization of the nodes into groups of similar connectivity profiles, then the nodes within the same group should have similar activities, and these activities should be close to the corresponding observable at any time. The validity of this assumption will of course depend on the properties of the chosen partition, especially on the number of groups and on the inter-node heterogeneity within each group.

We present two possible reductions that result from considering zeroth-order and first-order Taylor approximations of functions $f$ and $g$ (approximations of a larger order would require the use of non-linear observables~\cite[Annexe B]{Thibeault2020_master} and is beyond the scope of the present work). We call them the \emph{homogeneous} and the \emph{spectral} reductions, respectively, for reasons that will become clear below.

\paragraph{Homogeneous reduction.}
In what we dub the \emph{homogeneous reduction}, we approximate $f(x_i) \approx f(\cal{X}_\nu)$ and $g(x_i,x_j) \approx g(\cal{X}_\nu,\cal{X}_\rho)$ whenever $i \in G_\nu$ and $j \in G_\rho$. Doing so immediately transforms Eq.~\eqref{Xnu dot} into an (approximate) closed form, regardless of the particular choice of the vectors $\bm{a}_1, \cdots, \bm{a}_n$. We choose to define these vectors as being homogenous over the different groups, that is,
\begin{equation}
  a_{\nu i} = \left\lbrace 
  \begin{array}{ll}
    1/m_\nu & \text{ if } i \in G_\nu, \\
    0 & \text{ otherwise}.
  \end{array} \right.
\end{equation}
The approximate reduced dynamics becomes
\begin{equation}
  \begin{array}{lll}
    \cal{\dot{X}}_\nu 
    &\approx& \displaystyle 
    f(\cal{X}_\nu)
    + \sum \limits_{\rho=1}^n \cal{W}_{\nu \rho} \, g(\cal{X}_\nu, \cal{X}_\rho), \\
  \end{array}
  \label{Xnu dot homogeneous}
\end{equation}
where
\begin{equation}
  \cal{W}_{\nu \rho} := \sum \limits_{\substack{ i \in G_\nu \\ j \in G_\rho}} a_{\nu i} w_{ij}
    = \frac{1}{m_\nu} \sum \limits_{i \in G_\nu} k_i^\rho
\label{calW nu rho}
\end{equation}
is the weighted in-degree coming from nodes in $G_\rho$, averaged over nodes in $G_\nu$ according to $\bm{a}_\nu$ (see Proposition \ref{prop homogeneous} in section~\nameref{sec: app closing the dynamics} of the SI for details). Notice that Eq.~\eqref{Xnu dot homogeneous} is analogous in form to the original dynamics. The matrix $\bm{\cal{W}} = (\cal{W}_{\nu \rho})_{\nu, \rho}$ is the \emph{reduced adjacency matrix}: a weighted matrix of interactions among the observables in the approximate reduced system.

\paragraph{Spectral reduction.}
In what we call the \emph{spectral reduction}, we go one step further and approximate $f$ and $g$ by first-order Taylor polynomials around the appropriate observables:
\begin{subequations}
\begin{align}
     f(x_i) & \approx f(\cal{X}_\nu) + f'(\cal{X}_\nu) (x_i - \cal{X}_\nu) \\
     g(x_i,x_j) &\approx g(\cal{X}_\nu, \cal{X}_\rho) + g_1(\cal{X}_\nu, \cal{X}_\rho) (x_i - \cal{X}_\nu) \nonumber \\
     & \qquad + g_2(\cal{X}_\nu, \cal{X}_\rho) (x_j - \cal{X}_\rho) 
\end{align}
\end{subequations}
where $i \in G_\nu, j \in G_\rho$, and $g_1, g_2$ denote the partial derivatives of $g$ with respect to its first and second arguments, respectively. Substituting these expressions into Eq.~\eqref{Xnu dot} does not, however, yield a closed dynamics. As stated in Proposition \ref{prop spectral} (see section~\nameref{sec: app closing the dynamics} of the SI), the reduction vectors must fulfill the following conditions to close the system:
\begin{subequations}
\begin{align}
    \bm{K}_{\nu \rho} \bm{\widehat{a}}_\nu & = \mu_{\nu \rho} \bm{\widehat{a}}_\nu \label{comp eq K} \\
    \bm{W}_{\nu \rho}^T \bm{\widehat{a}}_\nu & = \lambda_{\nu \rho} \bm{\widehat{a}}_\rho \ , \label{comp eq W}
\end{align}
\end{subequations}
for $\nu, \rho \in \{1, \cdots, n\}$.
The vector $\bm{\widehat{a}}_\nu$ is defined by the components of $\bm{a}_\nu$ that correspond to nodes within group $G_\nu$ (the other elements are 0 by definition):
\begin{equation}
  \bm{\widehat{a}}_\nu = ( \widehat{a}_{\nu i} )_{i=1}^{m_\nu} \in \mathbb{R}^{m_\nu}, \hspace{0,4cm}
  \widehat{a}_{\nu i} := a_{\nu \, p_\nu(i)}, \hspace{0,4cm}
  p_\nu(i) := i + \sum \limits_{s=1}^{\nu-1} m_s,
  \label{def a hat}
\end{equation}
assuming, again, that nodes have been reordered according to their group membership. We call $\bm{\widehat{a}}_1, \cdots, \bm{\widehat{a}}_n$ the \emph{partial reduction vectors}. The matrices $\bm{\mu} = (\mu_{\nu \rho})_{\nu,\rho}$ and $\bm{\lambda} = (\lambda_{\nu \rho})_{\nu,\rho}$, both of dimension $n \times n$, are sets of parameters to be determined.  We note that Eqs.~\eqref{comp eq K} and \eqref{comp eq W} need to be imposed whenever the $g$ function varies with its first and its second arguments, respectively. This means, for example, that reducing a system in which $g(x,y)=g(y)$ does not require Eq.~\eqref{comp eq K} to be fulfilled.

\begin{figure}[ht]
  \centering
  \includegraphics[width=1\linewidth]{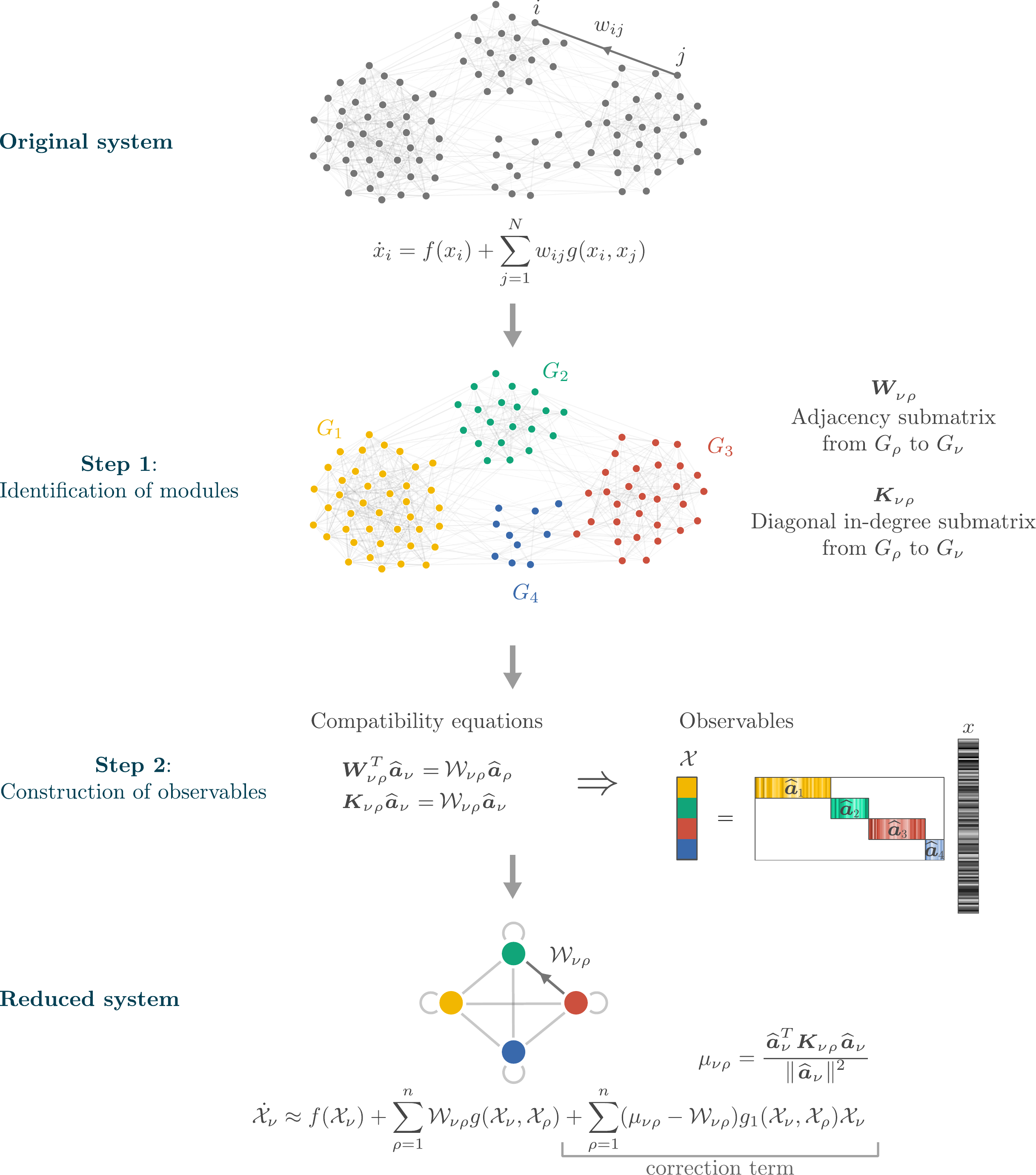}
  \caption{\small Method schematics of the spectral reduction. First, the nodes are partitioned into assortative or disassortative groups so that nodes in the same group share similar connectivity properties. In this case, $n=4$ groups were defined. Group-to-group adjacency and in-degree submatrices are defined from the original adjacency matrix. Second, these matrices are used to solve the compatibility equations on the partial reduction vectors $\bm{\widehat{a}}_1, \cdots, \bm{\widehat{a}}_n$ and the  matrix $\bm{{\cal W}} = ({\cal W}_{\nu \rho} )_{\nu,\rho=1}^n$. Once the compatibility equations are solved, the observables are constructed as the scalar product of the reduction vectors and the activity vector $\bm{x}$. The final approximate reduced dynamics on the observables is analogous to the original dynamics---with  $\bm{{\cal W}}$ acting as a reduced adjacency matrix---and can also incorporate a correction term.
  }
  \label{fig spectral red}
\end{figure}

As in Ref.~\cite{thibeault_threefold_2020}, we refer to conditions~\eqref{comp eq K} and \eqref{comp eq W} as the \emph{compatibility equations}. They relate the reduction vectors with the adjacency and weighted in-degree matrices, independently of the functions that define the node dynamics. It is therefore solely the structure of interactions in the network that shapes and constrains the construction of the observables in our approach.
 
As detailed in the proof of Proposition \ref{prop spectral} of the SI (section~\nameref{sec: app closing the dynamics}), once the compatibility equations are fulfilled, the parameters $\lambda_{\nu \rho}$ and $\mu_{\nu \rho}$ coincide with the group-to-group averaged in-degree defined by Eq.~\eqref{calW nu rho},
\begin{subequations}
\begin{equation}
  \mu_{\nu \rho} = \cal{W}_{\nu \rho}
  \label{mu nu rho}
\end{equation}
\begin{equation}
  \lambda_{\nu \rho} = \cal{W}_{\nu \rho},
  \label{lambda nu rho}
\end{equation}
\end{subequations}
and the approximate temporal evolution of the $\nu$-th observable is given by
\begin{equation}
  \cal{\dot{X}}_\nu 
  \approx \displaystyle f(\cal{X}_\nu)
  + \sum \limits_{\rho = 1}^n \cal{W}_{\nu \rho} \, g( \cal{X}_\nu, \cal{X}_\rho ).
  \label{approx reduced dynamics}
\end{equation}
As in the homogeneous reduction, this dynamics preserves the form of the original dynamics and includes a $n \times n$ \emph{reduced adjacency matrix} $\cal{\bm{W}} = (\cal{W}_{\nu \rho})_{\nu,\rho}$ between the different observables, defined by Eq.~\eqref{calW nu rho}.

Thus, in the spectral reduction, mapping the original dynamics into the reduced dynamics requires (a) solving the compatibility equations to find the reduction vectors $\bm{a}_1, \cdots, \bm{a}_n$, and (b) computing the reduced adjacency matrix $(\cal{W}_{\nu \rho})_{\nu, \rho}$ from the reduction vectors. The second step is straightforward but not the first one. Solving the compatibility equations can be problematic because, except in very particular cases, these equations cannot be fulfilled simultaneously. In the next section we propose a method for finding an approximate solution, which takes into account the spectral properties of the group-to-group weighted in-degree and adjacency matrices. This is why we call it the \emph{spectral reduction}. Fig.~\ref{fig spectral red} shows a schematics of the spectral reduction process.

\subsection{Solving the compatibility equations}
The spectral reduction requires solving the compatibility equations~\eqref{comp eq K}~and~\eqref{comp eq W} to determine the reduction vectors. Equation~\eqref{comp eq K} is an eigenvalue-eigenvector equation for the partial reduction vector $\bm{\widehat{a}}_\nu$. Equation~\eqref{comp eq W} is more involved because it includes a crossed dependency between different partial reduction vectors.
Corollary~\ref{corollary 1} in the SI (section~\nameref{sec: app equivalence comp eqs}) provides a strategy to solve Eqs.~\eqref{comp eq K}~and~\eqref{comp eq W} when the original adjacency matrix is positive that can be sumarized as follows.
If we assume the partial reduction vectors $ \bm{\widehat{a}}_1, \cdots, \bm{\widehat{a}}_n$ to be strictly positive, it is possible to find an equivalent form for Eq.~\eqref{comp eq W} without crossed dependencies. The resulting set of \emph{decoupled} compatibility equations reads
\begin{subequations}
\begin{align}
    \bm{K}_{\nu \rho} \bm{\widehat{a}}_\nu = \mu_{\nu \rho} \bm{\widehat{a}}_\nu
    \label{comp eq K decoupled} \\
    \bm{W'}_{\nu \rho} \bm{\widehat{a}}_\nu = \lambda_{\nu \rho}' \bm{\widehat{a}}_\nu
    \label{comp eq W decoupled}
\end{align}
\end{subequations}
for $\nu, \rho \in \{1,\cdots,n\}$, where
\begin{subequations}
\begin{align}
  \lambda_{\nu \rho}' & :=
    \begin{cases}
      \lambda_{\nu \nu} & \text{ if } \nu = \rho, \\
      \lambda_{\nu \rho} \lambda_{\rho \nu}\phantom{m} & \text{ if } \nu \neq \rho,
    \end{cases}\\
  \bm{W'}_{\nu \rho} & :=
    \begin{cases}
      \bm{W}_{\nu \nu}^T & \text{ if } \nu = \rho, \\
      \bm{W}_{\rho \nu}^T \bm{W}_{\nu \rho}^T & \text{ if } \nu \neq \rho.
    \end{cases}
  \end{align}
\end{subequations}
The decoupled form~\eqref{comp eq W decoupled} for fixed $\nu$ and variable $\rho$ is easier to treat because it consists of a set of $n$ equations on $\bm{\widehat{a}}_\nu$ and $\{\lambda_{\nu \rho}'\}_\rho$ that we can try to solve independently for each $\nu$. These equations state that $\bm{\widehat{a}}_\nu$ has to be, simultaneously, the dominant eigenvector of a collection of $n$ matrices (the dominant condition being a consequence of the Perron-Frobenius Theorem). Equation~(\ref{comp eq K decoupled}) states that $\bm{\widehat{a}}_\nu$ should also be an eigenvector of the $n$ diagonal matrices $\bm{K}_{\nu 1}, \cdots, \bm{K}_{\nu n}$.

Let us first address the problem of solving Eq.~\eqref{comp eq W decoupled} for a fixed $\nu$ and with $\rho$ ranging in $\{1, \cdots, n\}$. In general, the equations in this set are not simultaneously solvable because, except in very particular cases, the matrices involved do not share the dominant eigenspace.
To find an approximate solution, we first assume that the scalars $\{ \lambda_{\nu \rho}' \}_\rho$ are the dominant eigenvalues of the matrices involved (this would be the case if the equations could be solved exactly). If we relax the requirement of $\bm{\widehat{a}}_\nu$ having positive entries, our goal is to find a vector $\bm{\widehat{a}}_\nu$ such that $\sum \limits_{i=1}^{m_\nu} \widehat{a}_{\nu i} = 1$ and which minimizes the sum of the corresponding quadratic errors,
\begin{equation}
  E( \bm{\widehat{a}}_\nu ) := \Vert \bm{W'}_{\nu 1} \bm{\widehat{a}}_\nu - \lambda'_{\nu 1} \bm{\widehat{a}}_\nu \Vert^2 + \cdots +  \Vert \bm{W'}_{\nu n} \bm{\widehat{a}}_\nu - \lambda'_{\nu n} \bm{\widehat{a}}_\nu \Vert^2.
\end{equation}

The solution to this problem is presented in Proposition~\ref{prop 3} (SI, section~\nameref{sec: app solution comp eqs W}). To find a solution vector lying in the subspace spanned by a collection of $r$ vectors $\bm{u}_1, \cdots, \bm{u}_r \in \mathbb{R}^{m_\nu}$ with $\sum_{i=1}^{m_\nu} [\bm{u}_s]_i = 1$ for all $s$, 
one has to determine $\bm{y} = (x_1, \cdots, x_r, K)$ that solves the linear equation
\begin{align}
  \bm{\hat{C} y} =  ( 0, \cdots, 0, 1 )^T,
  \label{linear system alpha}
\end{align}
where
$\bm{\hat{C}} := \left(
\begin{array}{c | c} 
\bm{C} & - \bm{1} \\
\hline
\bm{1}^T & 0
\end{array}
\right),$
$\bm{1} = (1, \cdots, 1)^T$ and $\bm{C} = (c_{st})_{s,t}$ is the $r \times r$ matrix given by
\begin{equation}
  c_{st} := \sum \limits_{\rho=1}^n \langle \bm{W'}_{\nu \rho} \bm{u}_s - \lambda'_{\nu \rho} \bm{u}_s,  \bm{W'}_{\nu \rho} \bm{u}_t - \lambda'_{\nu \rho} \bm{u}_t \rangle.
\end{equation}
The solution is then
\begin{equation}
  \bm{\widehat{a}}_\nu = x_1 \bm{u}_1 + \cdots + x_r \bm{u}_r.
\end{equation}
This equation might have different solutions, but the error associated to all of them is the same (see related discussion after Proposition~\ref{prop 3} in the SI for details).

The above procedure allows us to find a solution that is restricted to the subspace spanned by $\bm{u}_1, \cdots, \bm{u}_r \in \mathbb{R}^{m_\nu}$. A solution that is not restricted to a particular subspace is obtained when $r=m_\nu$ and $\bm{u}_1, \cdots, \bm{u}_r$ is the canonical basis of $\mathbb{R}^{m_\nu}$. This is the \emph{optimal} solution, i.e., the one with the smallest error. We have observed that the optimal solution is extremely close to the subspace spanned by the dominant eigenvectors of matrices $\bm{W'}_{\nu 1}, \cdots, \bm{W'}_{\nu n}$. This suggests that a very good approximation to the optimal solution is obtained when $r=n$ and $\bm{u}_1, \cdots, \bm{u}_r$ are the dominant eigenvectors of these matrices (see Fig.~\ref{fig optimal vs restricted solution} and section~\nameref{sec: app solution comp eqs W} of the SI for details). Finding the solution in the subspace spanned by the dominant eigenvectors has clear computational advantages when $n \ll m_\nu$.
In the results shown here we restrict the solution to Eq.~\eqref{linear system alpha} to the subspace spanned by the dominant eigenvectors. We have observed that, despite not explicitly requiring $\bm{\widehat{a}}_\nu$ to be a positive vector, this is so in all the cases studied.

We have addressed the problem of solving Eq.~\eqref{comp eq W} but not Eq.~\eqref{comp eq K}.  
Assuming again that $\bm{W}$ is a positive matrix and that we want the reduction vectors to be positive, Eqs.~\eqref{comp eq K}~and~\eqref{comp eq W} are equivalent to Eqs.~\eqref{comp eq K decoupled}~and~\eqref{comp eq W decoupled}; for a fixed $\nu$, these are eigenvector-eigenvalue equations for the vector $\bm{\widehat{a}}_\nu$. Thus, it could be tempting to try to find an approximate solution following the strategy proposed above just by adding Eqs.~\eqref{comp eq K decoupled} for all $\rho$ to the list of eigenvector-eigenvalue equations that involve $\bm{\widehat{a}}_\nu$. However, contrary to what happens with Eqs.~\eqref{comp eq W decoupled}, we do not have a criterion to choose the corresponding scalars $\mu_{\nu \rho}$. In principle, any of the diagonal entries of matrix $\bm{K}_{\nu \rho}$ could be chosen as $\mu_{\nu \rho}$, but the resulting error and solution vector $\bm{\widehat{a}}_\nu$ could be very different depending on this choice.

To solve this issue we apply the following procedure: for a fixed $\nu$, we first approximately solve Eqs.~\eqref{comp eq W decoupled} to determine $\bm{\widehat{a}}_\nu$. Once this vector is specified, for every $\rho$ we compute the scalar $\mu_{\nu \rho}$ that minimizes the quadratic error associated to Eq.~\eqref{comp eq K decoupled}, which is given by
\begin{equation}
  \mu_{\nu \rho} 
  = \frac{ \bm{\widehat{a}}_\nu^T \bm{K}_{\nu \rho} \bm{\widehat{a}}_\nu}
  { \| \bm{\widehat{a}}_\nu \|^2}
  \label{mu nu rho minimal error}
\end{equation}
(see Lemma \ref{lemma min quadratic error} in the SI). In doing so, we no longer assume that Eq.~\eqref{mu nu rho} holds, and this leads to a reduced dynamics that incorporates a correction factor which depends on matrix $\bm{\mu} = (\mu_{\nu \rho})_{\nu,\rho}$:
\begin{multline}
\cal{\dot{X}}_\nu 
\approx \displaystyle  f(\cal{X}_\nu) 
+ \sum \limits_{\rho = 1}^n \cal{W}_{\nu \rho} \, g(\cal{X}_\nu, \cal{X}_{\rho})  \\ \displaystyle
+ \sum \limits_{\rho = 1}^n \left( \mu_{\nu \rho} - \cal{W}_{\nu \rho} \right) g_1(\cal{X}_\nu, \cal{X}_{\rho}) \cal{X}_\nu
\label{approx dynamics + corr mu}
\end{multline}
(see section~\nameref{sec: app corrected reduced system} of the SI for details).

Notice that the reduced dynamics described by Eq.~\eqref{approx dynamics + corr mu} reduces to the one presented earlier in Eq.~\eqref{approx reduced dynamics} when the compatibility equations that involve the weighted in-degree matrix are solved exactly or when the $g$ function does not depend on its first argument. Also, when the node partition is such that the connectivity properties of nodes in the same group are very similar, we can expect each matrix $\bm{K}_{\nu \rho}$ to approximately be a multiple of the identity matrix, and in this case conditions~\eqref{comp eq K decoupled}~and~\eqref{mu nu rho} are automatically fulfilled for any $\bm{\widehat{a}}_\nu$. This suggests that the reduced dynamics~\eqref{approx reduced dynamics} is quite accurate when the groups are composed of nodes of similar connectivity properties, which is one of the main assumptions behind our reduction methods.

\section{Exploring the homogeneous and spectral reductions}
So far, we have presented two methods for reducing a given $N$-dimensional dynamical system on a network into an $n$-dimensional one whose variables, the observables, represent weighted averages of the node activities of the original network. We assumed that the nodes in the original network are organized into $n$ groups of similar connectivity properties. In the reductions, the observables are constructed so that each of them represents the activity within each of these groups.

The accuracy of the reduction is therefore expected to strongly depend on the number of groups $n$ and on the precise arrangement of the nodes into the different groups. Defining the groups and anticipating what $n$ should be to get a proper reduction is, however, a difficult task and there is no clear method to that end. A large repertoire of community detection algorithms on networks has been developed in the last years, which include spectral-based methods, algorithms that use information theory analysis, and Bayesian inference methods that fit the input network to modular graph models~\cite{Doreian2020, Fortunato2016, peixoto2021descriptive}, to cite only some examples. The issue of detecting such groups, despite being necessary for our reductions to be accurate, is not central to the present work. When studying the dimension-reduction methods that we have described earlier we will assume that the networks have been previously analyzed and the main communities have been already detected. We will nonetheless propose ways to refine a given node partition to obtain a richer and more accurate reduced system when the node heterogeneity within the given groups is too large. This will be a fruitful strategy when dealing with networks whose connectivity is highly heterogeneous.

We expect a proper reduction to be able to capture some salient features of the original dynamical system. In systems that are at equilibrium, such a feature can be the system's response to perturbations of some structural or dynamical parameters. The detection of bifurcation points, for example, can be crucial in anticipating global shifts in the system's behavior under such perturbations. We would like the reduced dynamics to capture these critical parameters even if the precise activity at equilibrium is not necessarily obtained with high accuracy.

In what follows, we study the ability of the two proposed reduction strategies to predict bifurcation diagrams of different dynamical systems as we vary the overall strength of all the interactions in the network. As our method for solving the compatibility equations requires the adjacency matrix to be positive, in this work we restrict ourselves to positive matrices. We can easily transform a non-negative weighted or binary adjacency matrix into a positive one by simply assuming that the missing interactions are arbitrarily weak.

\subsection{Examples of node dynamics}
The proposed methods for dimension reduction can be applied to a node dynamics defined by Eq.~\eqref{system} for arbitrary functions $f$ and $g$ of class $\mathscr{C}^1$. In this paper we have chosen $f$ and $g$ so as to model three types of dynamics on networks: a form of neuronal dynamics, a model of infectious disease spreading and an ecological dynamics.

\paragraph{Neuronal dynamics.}
We take as an example of neuronal dynamics Hopfield's continuous model~\cite{hopfield_neurons_1984}. Each node in the network represents a neuron that receives and projects inputs to the other neurons via synaptic connections. The node activity $x_i$ represents the mean membrane potential 
of neuron $i$ and evolves according to
\begin{equation}
  \dot{x}_i = -x_i + \sum \limits_{j=1}^N w_{ij} \, g(x_j),
  \label{system WC}
\end{equation}
where $g$ is a sigmoid function of one variable only which transforms the 
potential $x_j$ of the presynaptic neuron $j$ into its output $g(x_j)$ (for example, its firing rate). Thus, in this case we have $g(x,y) = g(y)$. We specifically take the function $g$ to be
\begin{equation}
  g(y) = \frac{1}{1 + \exp(-\tau ( y - \mu) )},
\end{equation}
where $\tau$ and $\mu$ are two parameters that control the maximal slope of $g$ and its location. Hopfield's continuous model is closely related to several well-known models of neuronal activity on networks, such as the Wilson-Cowan and Grossberg models \cite[Sec.~6.C]{grossberg1988nonlinear} or the firing-rate model \cite[p.~360]{vogels2005neural}.

\paragraph{Infectious dynamics.}
The SIS (susceptible-infected-susceptible) model aims at describing the spread of a disease in a network of contacts. Each node can be in two possible states: susceptible or infected. The node state stochastically evolves in time according to the states of the nodes it is in contact with: a susceptible node becomes infected at a rate $\lambda$ times the number of infected contacts and an infected node becomes susceptible again at a constant rate of 1. It is possible to define a mean-field version of the model that specifies the temporal evolution of the probability $x_i$ of node $i$ being infected in the contact network~\cite[Sec.~V.A.2]{pastor2015epidemic}:
\begin{equation}
  \dot{x}_i = -x_i + \gamma (1-x_i) \sum \limits_{j=1}^N w_{ij} \, x_j,
  \label{system SIS}
\end{equation}
where $\gamma \geq 0$ is the normalized infection rate.

\paragraph{Ecological dynamics.}
We consider a network of interacting species in a given ecosystem. If $x_i$ represents the abundance of species $i$, the evolution of the species' abundances can be modeled by

\begin{equation}
  \dot{x}_i = B + x_i \left( 1-\frac{x_i}{K} \right) \left( \frac{x_i}{C} -1 \right) 
  + \sum \limits_{j=1}^N w_{ij} \frac{x_i x_j}{D + E x_i + H x_j},
  \label{system ecology}
\end{equation}
where $B$ is a constant migration rate, $K>0$ is a carrying capacity and $C>0$ is the minimum abundance of species $i$ for it to grow~\cite{holland_population_2002}. The parameters $D, E, H$ shape the inter-species coupling dynamics.
We assume that the adjacency matrix is positive, so the dynamics is that of a mutualistic network.

\subsection{Exact versus reduced bifurcation diagrams for networks with block structure}

We first assess the extent to which the homogeneous and spectral reduction methods are able to reflect the system's sensitivity to parameter changes in random networks with known block structure. For this, we homogeneously vary the magnitude of all the interactions so as to force the system to transition between bistable regimes and regimes characterized by a single equilibrium point. By modifying the interaction strengths back and forth and integrating the system towards equilibrium, we can capture a bifurcation diagram that reflects these transitions (see section~\nameref{sec: app bif diagrams} of the SI for further details). 

To compare the reduced and the exact bifurcation diagrams, we plot a weighted average of the $n$ observables at equilibrium for the exact and reduced systems, $\langle \cal{X} \rangle$, as a function of the average weighted in-degree of the reduced system, $\langle \cal{K} \rangle$. The magnitude $\langle \cal{X} \rangle$ is the observables' average weighted by group size and reflects the overall state of the system:
\begin{equation}
\langle \cal{X} \rangle := \frac{1}{N} \sum \limits_{\nu = 1}^n m_\nu \, \cal{X}_\nu.
\label{av observable}
\end{equation}
The parameter $\langle \cal{K} \rangle$ is defined by
\begin{equation}
\langle \cal{K} \rangle := \frac{1}{N} \sum \limits_{\nu = 1}^n m_\nu \, \cal{K}_\nu,
\hspace{1cm}
\cal{K}_\nu := \sum \limits_{\rho = 1}^n \cal{W}_{\nu \rho}.
\label{av degree}
\end{equation}

\paragraph{Homogeneous networks.}

We start by considering random networks constructed according to the directed version of the stochastic block model (SBM)~\cite{Holland1983}, in which nodes are arranged into $n$ modules and binary connections appear independently with probabilities that depend on the node membership. By varying these probabilities we can create a full range of network structures, from assortative ones, in which nodes are densely connected to nodes in their same module, to dissortative networks, in which interactions within nodes in the same module are rare.

Figure~\ref{fig hom wc} compares the two reductions when we take the whole network as a single group ($n=1$) and when the node partition is that of the true communities in the network. When the number of communities is larger than 1, neither the homogeneous nor the spectral reductions are able to reproduce the correct bifurcation diagram with $n=1$ but both of them yield very accurate results when the true communities are provided and $n$ is increased accordingly. The two methods exhibit a very similar performance in this case. The reason is that the SBM (in the dense regime and when the network's size is large) tends to generate quite homogeneous networks, in which there is small variability in terms of connectivity among the nodes that are in the same group (see Fig.~\ref{fig degrees}A). In general, this makes the homogeneous reduction enough for predicting the bifurcation diagram. This is not a universal principle, though: there are situations in which correctly identifying the communities might not guarantee an accurate prediction of the bifurcation points by the reduced dynamics. We will come back to this issue later.

\begin{figure*}[ht]
  \centering
  \includegraphics[width=\linewidth]{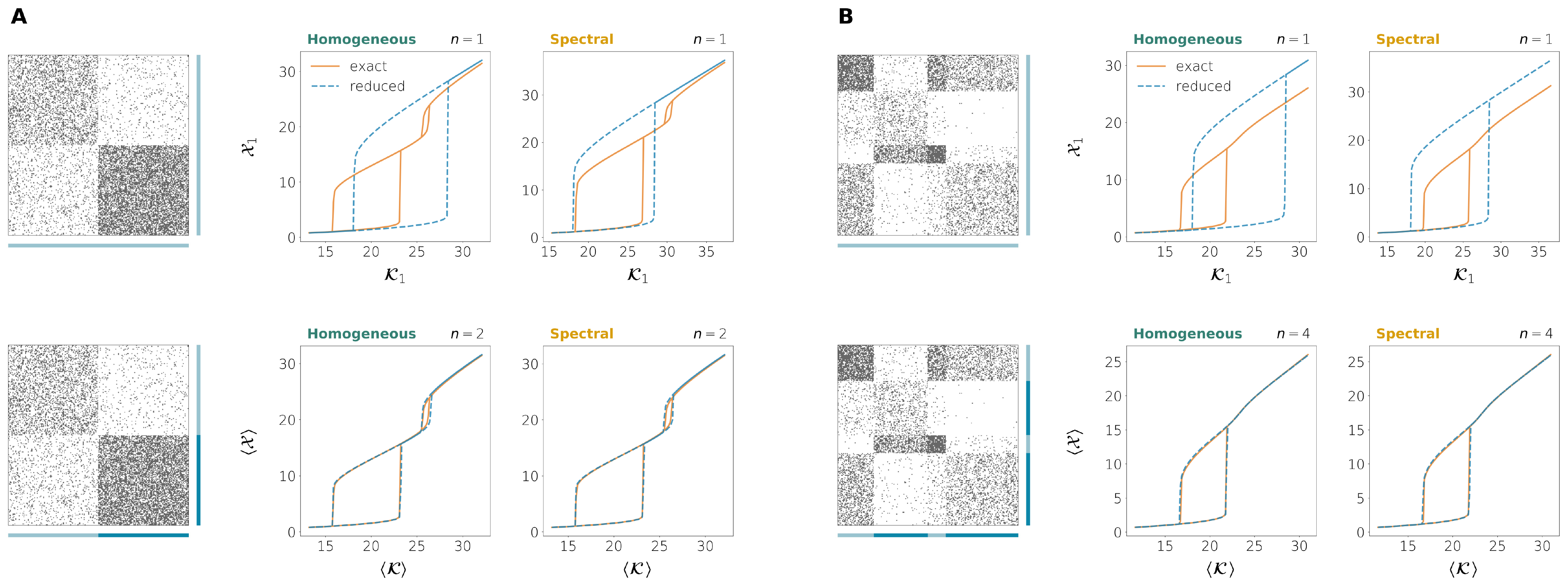}
  \caption{\small 
    Exact versus reduced bifurcation diagrams for homogeneous directed networks generated from the SBM and neuronal dynamics with $\tau = 0.3$, $\mu = 10$. The left panels show the adjacency matrices and the blue bars indicate the node partition used for the reductions. When the dimension of the reduction is $n>1$, the bifurcation diagrams show the value of the average observable at equilibrium defined by Eq.~\eqref{av observable}. The parameter on the x-axis is the average weighted in-degree of the reduced system defined by Eq.~\eqref{av degree}.
    \textbf{A}. Network on $N=200$ nodes and 2 communities of the same size in which the mean connection densities are $p_{11} = 0.3$, $p_{12} = 0.05$, $p_{21} = 0.1$, $p_{22} = 0.6$.
    \textbf{B}. Network on $N=200$ nodes and 4 communities with relative sizes 0.2, 0.3, 0.1, 0.4 and mean connection densities
      $p_{11} = 0.75$, $p_{12} = 0.05$, $p_{13} = 0.6$, $p_{14} = 0.3$,  
      $p_{21} = 0.1$, $p_{22} = 0.2$, $p_{23} = 0.03$, $p_{24} = 0.002$,  
      $p_{31} = 0.01$, $p_{32} = 0.5$, $p_{33} = 0.9$, $p_{34} = 0.05$,  
      $p_{41} = 0.35$, $p_{42} = 0.03$, $p_{43} = 0.1$, $p_{44} = 0.25$.
  }
  \label{fig hom wc}
\end{figure*}

\paragraph{Heterogeneous networks.}

We now analyze the performance of the two methods when the network has a known modular structure and there is also a large heterogeneity in the connectivity properties of nodes that are in the same module. Inspired by the Chung-Lu model~\cite{Chung2002a, Chung2002pnas}, we created a modified version of the directed SBM that can incorporate an important variability of in- and out-degrees within nodes belonging to the same community (see section~\nameref{sec: app heterogeneous networks} of the SI as well as Fig.~\ref{fig degrees}B). In this case, the homogeneous and the spectral methods show different performances, even when the true community structure is used to define the groups in the reductions. The spectral reduction tends to provide more accurate bifurcation diagrams. Yet, when the dimension $n$ is lower or equal to the number of communities, even the spectral method might yield results that are not quantitatively accurate for some dynamics. This is illustrated by the two upper rows of Fig.~\ref{fig het wc}, that show the performance of our reduction methods on heterogeneous networks with the same dynamics, the same number of communities and the same mean connection densities as in the homogeneous networks of Fig.~\ref{fig hom wc}.

This lack of accuracy is caused by the large heterogeneity among the nodes in the same group. In a way, the number of \emph{effective} groups in these networks is larger than the number of communities used to construct the network. Therefore, more refined partitions should be defined in order for the reduced dynamics to accurately predict the bifurcation diagrams.

To refine a given partition, we used a procedure whose goal is to divide the groups into subgroups so that the variability within nodes that are in the same subgroup is reduced. This variability could correspond to different connectivity attributes of nodes. As an example, we focused on the weighted in/out-degrees from/to the other groups (see section~\nameref{sec: app part refinement} of the SI). This allows us to create nested partitions from an original partition while progressively increasing the number of groups $n$ and, with it, the dimension of the reduced dynamics.

As shown in the lower rows of Fig.~\ref{fig het wc}, the partition refinements improve the quality of the reductions, especially for the spectral method. We can measure the reduction's quality by the root-mean-square-error (RMSE) between the true bifurcation diagram and the one obtained from simulating the reduced dynamics (Fig.~\ref{fig het dyn errors wc}A). Figure~\ref{fig het dyn errors wc}B~and~\ref{fig het dyn errors wc}C shows the RMSE as we increase the reduced dimension $n$.

Our results suggest that the spectral reduction is able to cope better with heterogeneities in the adjacency matrix. The comparison between these networks and their homogeneous counterparts of Fig.~\ref{fig hom wc} indicates that the effective dimension of the heterogeneous networks is larger, although it is still smaller compared to that of the original dynamics ($N=200$). We observe a similar trend in other network examples provided with different node dynamics (Fig.~\ref{fig het SIS}).

\begin{figure}[!t]
  \centering
  \includegraphics[width=\linewidth]{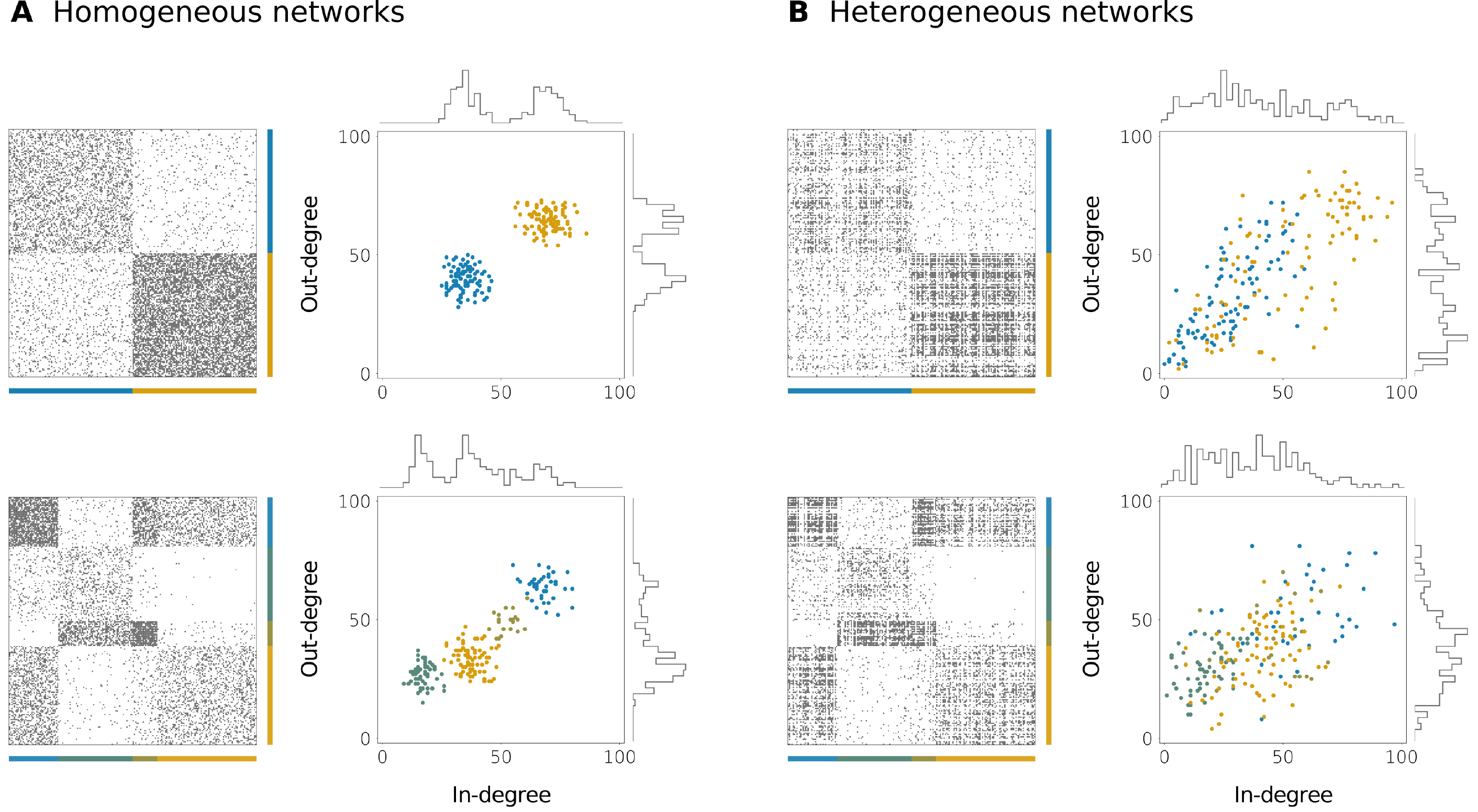}
  \caption{\small Adjacency matrices and in/out-degrees for homogeneous (\textbf{A}) and heterogeneous (\textbf{B}) directed networks. Each color corresponds to one community. The networks are analogous to those of Figs.~\ref{fig hom wc}~and~\ref{fig het wc}.}
  \label{fig degrees}
\end{figure}

\begin{figure*}[ht]
  \centering
  \includegraphics[width=\linewidth ]{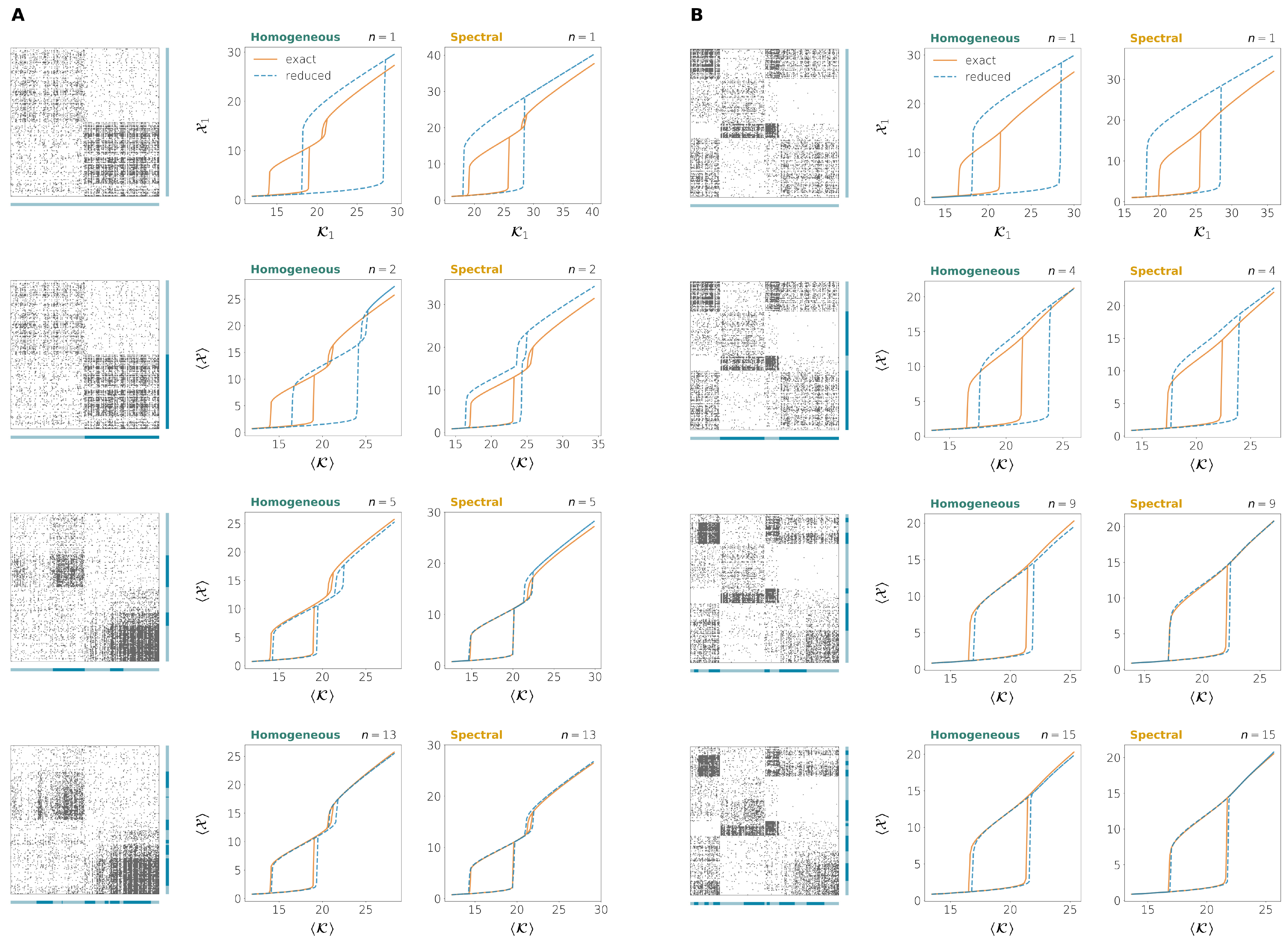}
  \caption{\small 
    Exact versus reduced bifurcation diagrams for heterogeneous directed networks and neuronal dynamics with $\tau = 0.3$, $\mu = 10$.
    \textbf{A}. Network on $N=200$ nodes and 2 communities with the same sizes and same mean connection densities as in Fig.~\ref{fig hom wc}A.
    \textbf{B}. Network on $N=200$ nodes and 4 communities with the same sizes and same mean connection densities as in Fig.~\ref{fig hom wc}B.
  }
  \label{fig het wc}
\end{figure*}

\begin{figure}[ht]
  \centering
  \includegraphics[width=\linewidth ]{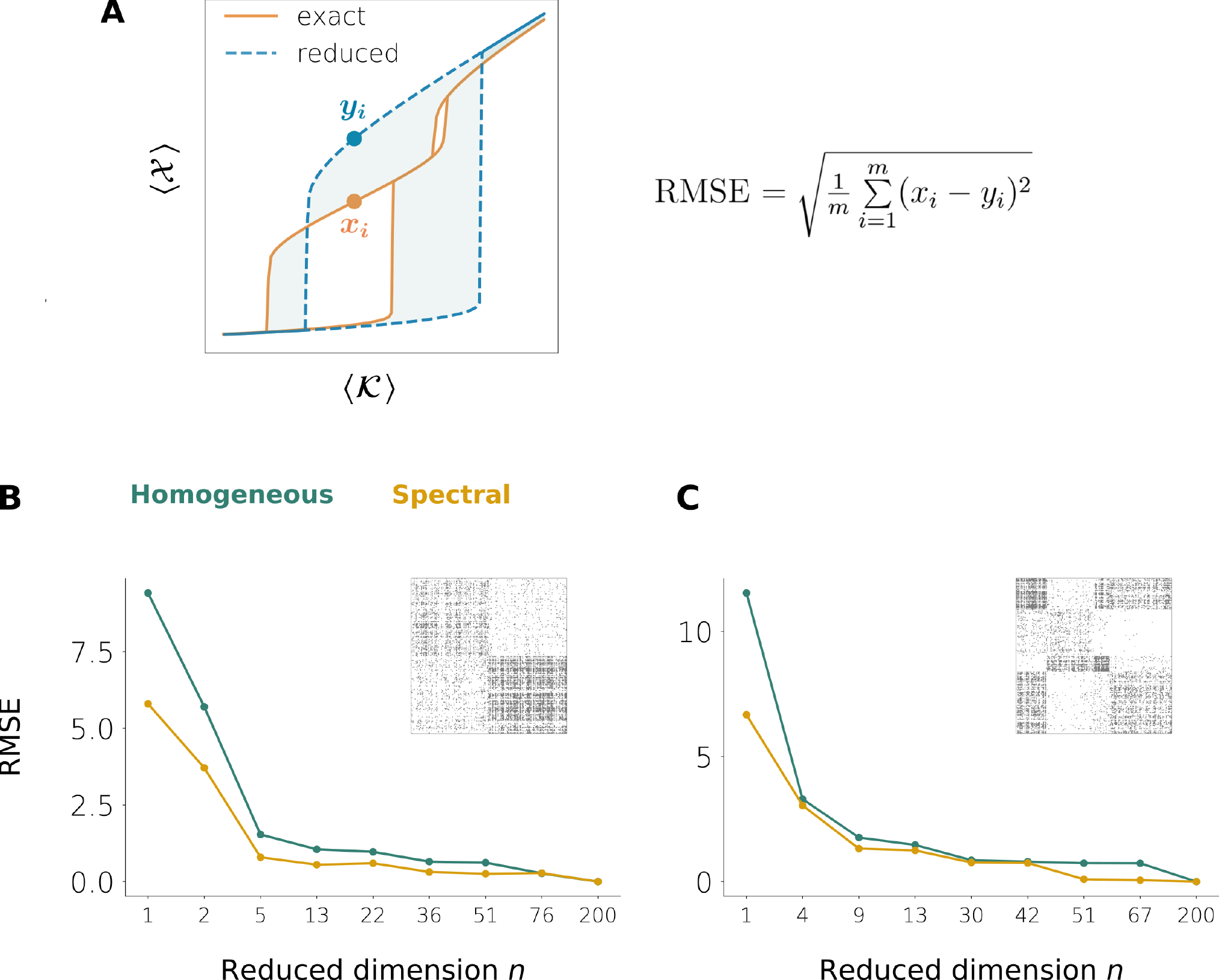}
  \caption{\small 
    Reduction error as the reduced dimension $n$ is increased.
    \textbf{A}. Root-mean-square-error (RMSE) between the exact and the reduced bifurcation diagrams as a measure of the reduction error.
    \textbf{B}. Network on $N=200$ nodes and 2 communities with the same mean connection densities and same dynamics as in Fig.~\ref{fig het wc}A.
    \textbf{C}. Network on $N=200$ nodes and 4 communities with the same mean connection densities and same dynamics as in Fig.~\ref{fig het wc}B.
  }
  \label{fig het dyn errors wc}
\end{figure}

\begin{figure}[ht]
  \centering
  \includegraphics[width=\linewidth ]{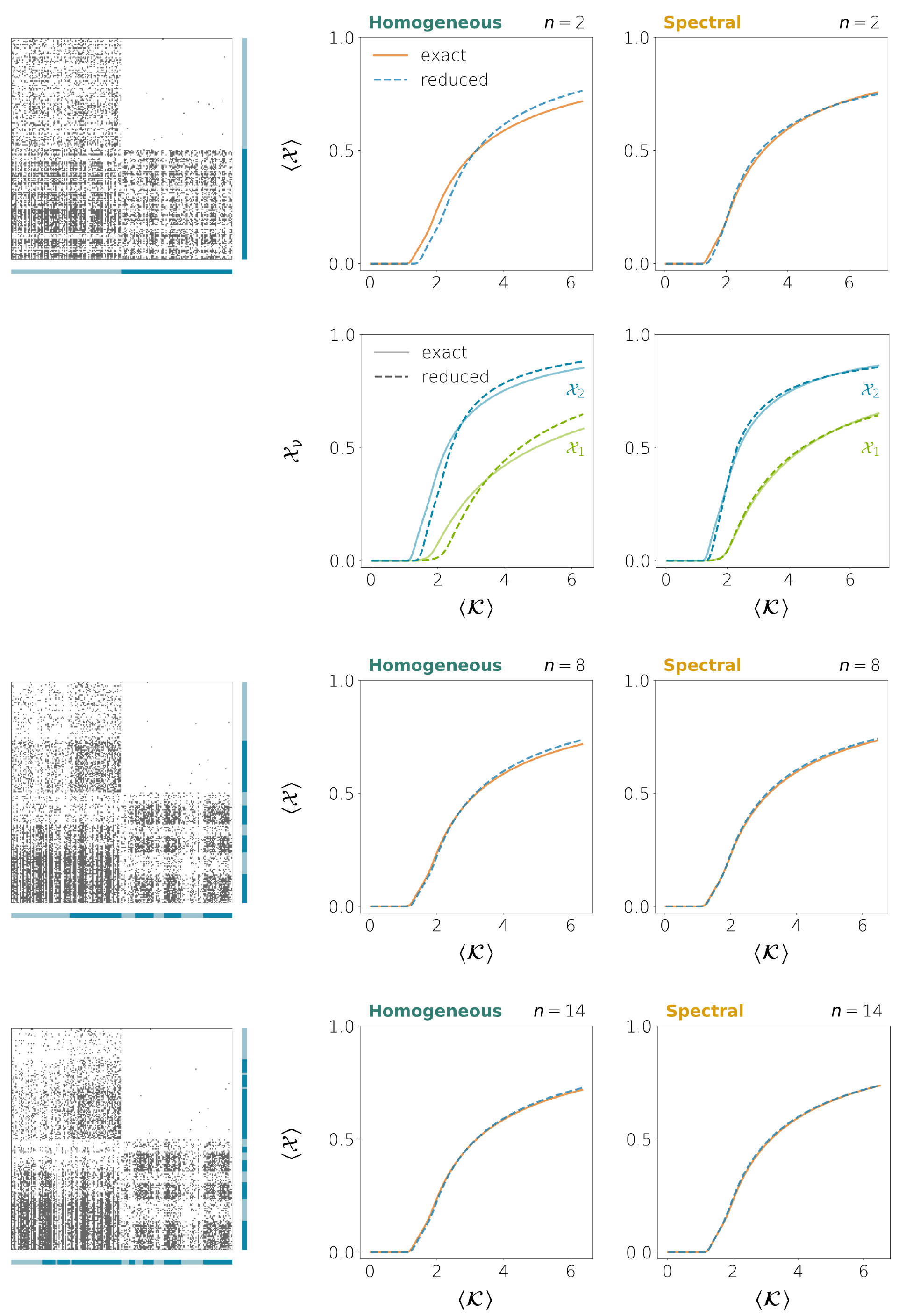}
  \caption{\small 
    Exact versus reduced bifurcation diagrams for a heterogeneous directed network and infectious dynamics with $\gamma = 1$.
    The network has $N=200$ nodes and 2 communities of the same size and with mean connection densities
      $p_{11} = 0.2$, $p_{12} = 0.001$,
      $p_{21} = 0.5$, $p_{22} = 0.3$.
  }
  \label{fig het SIS}
\end{figure}

\paragraph{Relevant heterogeneities in homogeneous networks.}

Even in homogeneous networks, there are situations in which correctly identifying the communities might not guarantee an accurate prediction of the bifurcation points by the reduced dynamics. Figure~\ref{fig hom n=2 eco} shows the bifurcation diagram corresponding to an ecological dynamics on a network composed of two communities. The connection density is relatively high within each group and also from group 1 to group 2 but not the other way around: there exist few connections from group 2 to group 1. When the initial state of the system is that of a low species abundance in both groups, the system stays in that state until the overall connection strength reaches a critical threshold that makes group 2 (the one with denser connectivity) jump to a high abundance state (Fig.~\ref{fig hom n=2 eco}A, blue continuous curves in the bottom plots). This first jump is well captured by the reduced dynamics (blue dashed curves). The reduced systems also predict that group 1 should remain in the low abundance state until a second threshold in connection strength is reached (green dashed curves). This threshold is nevertheless not well predicted because group 1 shifts to a high abundance state much earlier in the complete system (green continuous curves). The discrepancy between the exact and the reduced bifurcation diagrams is large for both reduction methods.

The source of the discrepancy is the heterogeneity in the input that nodes in group 1 receive from group 2: the majority of nodes in group 1 do not receive any input from group 2 but a few of them get input from exactly one node in group 2. In the original network, the input that this last set of nodes receives is enough for them to jump to a high abundance state right after the jump of group 2. Due to the high density within group 1, they recruit the other nodes in their group and the result is that the whole group shifts to a high abundance state right after group 2 does so. The reduced system with 2 observables (either homogeneous or spectral), however, only takes into account the \emph{average} input from group 2 to group 1. This average is not large enough to drive group 1 to the high abundance state unless the connection strengths are much larger.

This suggests that the \emph{deviations} from the average input are the cause of the mismatch between the original and the reduced system. In this example, a small deviation in the input received can change the whole state of the system, whereas the reduced system is not able to capture it because it relies on group input averages.

The problem can be solved by refining the partition until the inter-group input variability is small enough. We applied the same partition refinement described in the previous section and we found that a partition into a large number of groups ($n \approx 116$ for a network on $N=200$ nodes) is needed to properly capture the second transition (Fig.~\ref{fig hom n=2 eco}B). This is so because our refinement method is designed to reduce the difference between the maximal and minimal weighted degrees in a given group, regardless of the magnitude of these degrees (i.e., a degree difference of 1 is treated the same way when the degrees are of the order of 20 and when they are of the order of 1). The result is that a very fine refinement is needed to separate the nodes in group 1 that cause the whole group to jump.

But it is possible to define an \emph{ad hoc} partition that makes the second transition be well captured by the reduced system. If we separate from the rest the nodes in group 1 that have in-degree of 1 from group 2, we get a 3-group partition for which the reduced dynamics exhibits a reasonably good accuracy (Fig.~\ref{fig hom n=2 eco}C), similar to that obtained for the refinement with $n=116$. This illustrates that, in some systems, small heterogeneities in connectivity can make the reduced system to fail in predicting the true transitions but this might be solved by properly choosing the node partition.

\begin{figure*}[ht]
  \centering
  \includegraphics[width=\linewidth ]{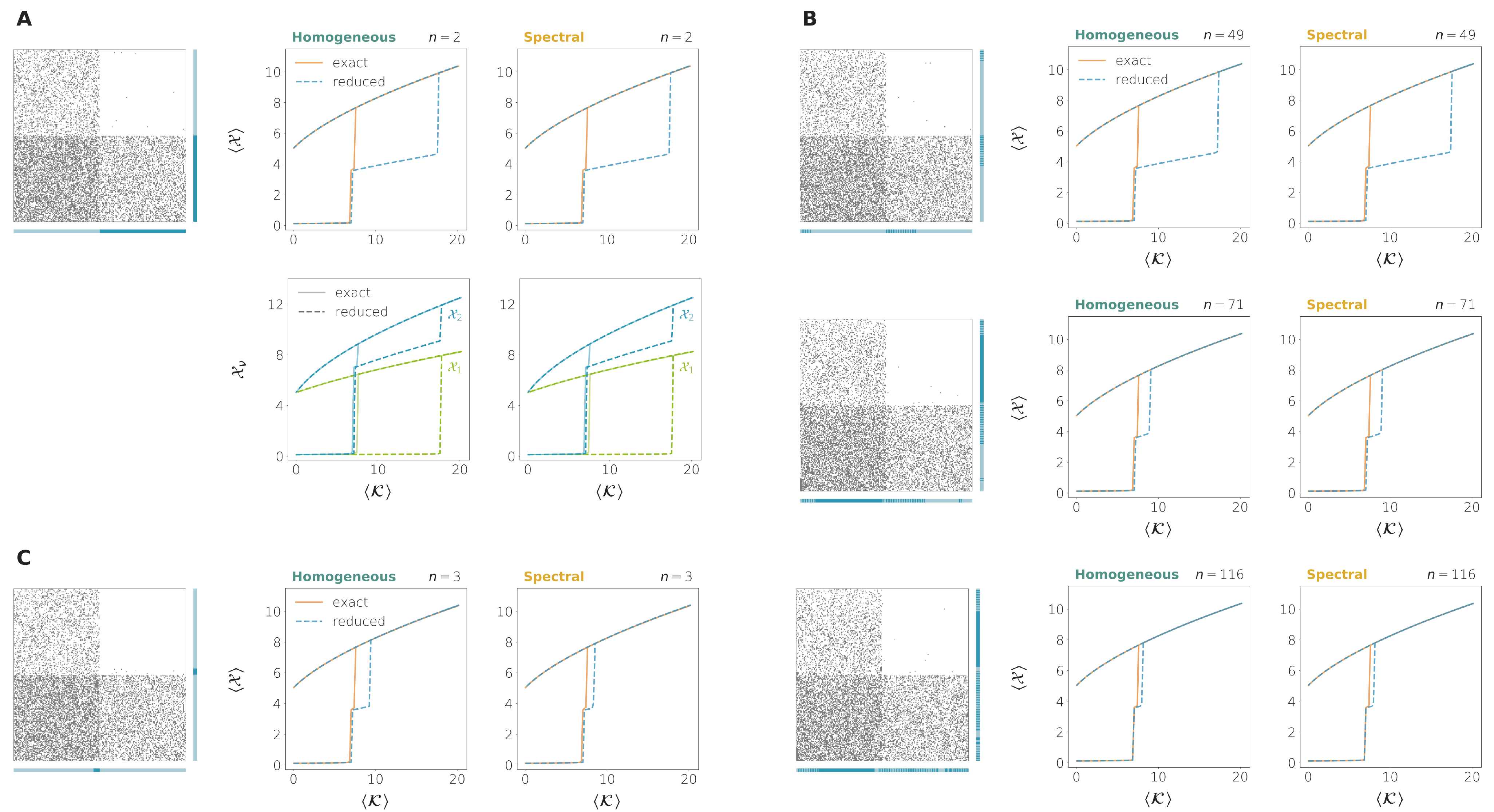}
  \caption{\small 
    Exact versus reduced bifurcation diagrams for a directed network on $N=200$ nodes generated from the SBM with 2 communities and ecological dynamics with $B = 0.1$, $C=1$, $K=5$, $D=6$, $E=0.9$, $H=0.1$. The community sizes and the mean connection densities between communities are the same as in Fig.~\ref{fig het SIS}.
    \textbf{A}. Bifurcation diagrams when the true communities are provided ($n=2$). The bottom panels show the exact versus reduced individual observables at equilibrium (red and blue curves).
    \textbf{B}. Bifurcation diagrams obtained as the original partition is progressively refined so as to reduce the intra-group weighted degree variability ($n=49$, $n=71$, $n=116$).
    {\textbf C.} Bifurcation diagrams when an \emph{ad hoc} partition is defined in which the nodes in group 1 that receive input from group 2 constitute a third group ($n=3$).
  }
  \label{fig hom n=2 eco}
\end{figure*}

\subsection{Robustness with respect to partition choice}

Until now we have analyzed the performance of the homogeneous and the spectral reductions for ``good'' partitions, that is, partitions that group nodes with similar connectivity properties. In particular, we have used partitions that correspond to the true modular organization of the network, together with successive refinements based on weighted in/out-degree variability. Doing so is necessary for the reductions to provide accurate results. However, in many situations we will not have access to any information regarding the presence of communities in the network. Instead, we will have to infer this information by analyzing the network structure, a task that can be problematic and might result in far from optimal partitions. Thus, a relevant question to be addressed is to what extent our reduction methods are sensitive to the partition choice. 

We analyzed the performance of our reduction methods as we randomly perturb a nearly optimal partition. Given an original partition $\cal{P}_0$, we select a pair of nodes at random and we flip their group membership. For $f \in [0,1]$, we repeat this process $\lfloor f N \rfloor$ times to get a new partition $\cal{P}_f$ that preserves the number of groups and the group sizes of $\cal{P}_0$ but which otherwise is a randomized version of $\cal{P}_0$ (Fig.~\ref{fig partition pert}A). We can repeat this many times to get an ensemble of perturbations of $\cal{P}_0$ for a given $f$ and then plot the average and the standard deviation of the discrepancy obtained according to each reduction method. Again, we define the discrepancy as the root-mean-square-error (RMSE) between the true bifurcation diagram and the one obtained from simulating the reduced dynamics (Fig.~\ref{fig het dyn errors wc}A).

Figures~\ref{fig partition pert}B~and~\ref{fig partition pert}C show the results for the 2-community networks (homogeneous and heterogeneous) explored in Figs. \ref{fig hom wc}A and \ref{fig het wc}A. We took as the original partitions the ones used in these figures for $n=2$ (homogeneous network) and $n=2$, $n=5$, $n=13$ (heterogeneous network). The plots show the average RMSE ($\pm$ standard deviation) relative to that of the spectral reduction for the original partition ($f=0$) as a function of $f$. Several conclusions can be derived from the results. First, the original partitions are close to be optimal because perturbing them results in larger errors on average. Second, the spectral method performs better than the homogeneous method, regardless of $f$. Third, the difference in performance between the two reduction methods is smaller for the homogeneous network, so the spectral method might be especially useful when dealing with heterogeneous networks. However, in the homogeneous network it is still preferable to use the spectral method, except if the true communities are perfectly identified ($f=0$, case in which the two methods exhibit the same performance). Finally, except in the heterogeneous network when $n=2$, the average error increase with respect to the case $f=0$ tends to be smaller for the spectral reduction, which suggests that this reduction is more robust to ill-posed node partitions.

\begin{figure*}[ht]
  \centering
  \includegraphics[width=\linewidth ]{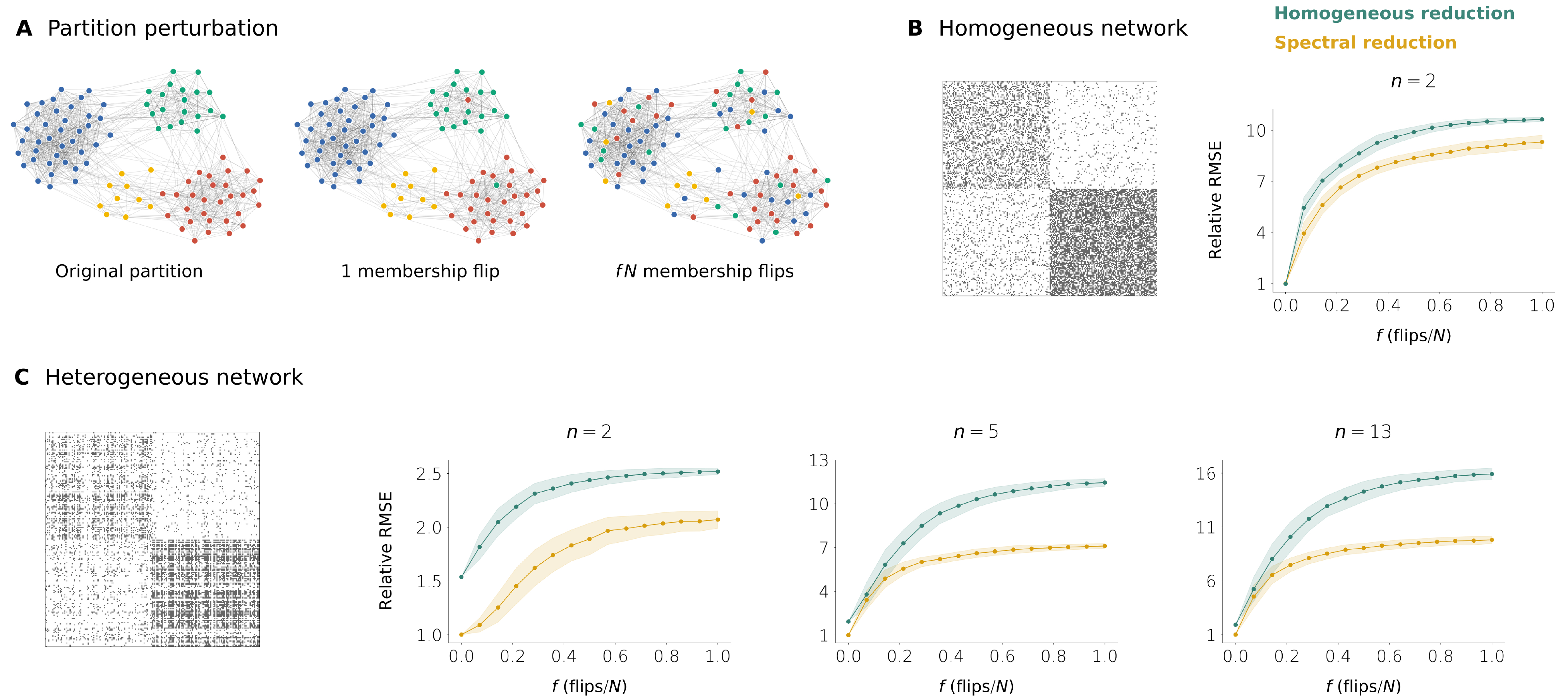}
  \caption{\small Sensitivity of the homogeneous and spectral reductions to partition perturbation.
    {\textbf A.} Schematics of the partition perturbation procedure.
    {\textbf{B--C.}} RMSE (relative to the RMSE of the spectral reduction when $f=0$) as a function of $f$ for the homogeneous (green) and the spectral (yellow) reductions. For each $f$, 300 random perturbations of the original partition were created. The lines show the average relative RMSE $\pm$ the standard deviation of the ensemble.
    {\textbf B.} Results for the homogeneous network of Fig.~\ref{fig hom wc}A when the original partition is the one used in that figure ($n=2$).
    {\textbf C.} Results for the heterogeneous network of Fig.~\ref{fig het wc}A. The original partitions are the ones shown in Fig.~\ref{fig het wc}A ($n=2$, $n=5$, $n=13$).
  }
  \label{fig partition pert}
\end{figure*}

\subsection{Dimension reduction on real networks}

We finally explore the performance of the homogeneous and spectral reduction methods when applied to dynamical systems on three networks obtained from real data.
We first consider the mutualistic ecological dynamics given by Eq.~\eqref{system ecology} together with a plant-pollinator network from the sub-alpine desert of Tenerife, in the Canary Islands~\cite{dupont_structure_2003}. It consists of 11 flowering plant species and 38 pollinator species (2 bird and 36 insect species), $N=49$. The authors assumed that an interaction exists between a plant and a pollinator whenever the pollinator had been observed probing for nectar or eating/collecting pollen from the plant. The resulting network of interactions is undirected and bipartite, as the imaginary plant-pollinator network depicted in Fig.~\ref{fig Dupont_ecology}A. 
When the interaction strength is large enough, the system has a single stable equilibrium state with large species abundances. But as the interactions are weakened, there is a transition to a bistable regime in which a  state characterized by very low species abundances is also possible. In the example shown here, this low abundance state is not an extinction state because we chose a positive migration rate $B$, meaning that even if the species went extinct, migration from other territories would make their numbers grow again.
In any case, the presence of such a low abundance state indicates that if the species' numbers are not large enough, the entire ecosystem can collapse into a state in which species can no longer benefit from the interaction with others and can only be maintained by migration. This kind of collapse might have catastrophic consequences for the ecosystem. 
The transition point is captured quite well by the spectral reduction, even when the reduction is 1-dimensional (Fig.~\ref{fig Dupont_ecology}C). The spectral method for $n=1$ also outperforms the degree-based 1-dimensional reduction defined by Gao et al.~\cite{gao_universal_2016} (Fig.~\ref{fig Dupont_ecology}B). As nodes are classified into groups of similar connectivity (by separating plants from pollinators first, $n=2$, and then by refining this partition), the tipping point is better approximated, the improvement being more evident in the homogeneous reduction, which provides poor results when $n$ is not large enough.

The second example is an undirected and binary social network based on Facebook contacts~\cite{maier_cover_2017} where $N=362$. To classify the nodes into groups of similar connectivity properties, we used the \texttt{minimize\underline{\space}blockmodel\underline{\space}dl} algorithm from the \texttt{graph-tool} Python library, which fits a given network into a stochastic block model~\cite{peixoto_efficient_2014}. The algorithm detected 15 communities and this partition was then refined so as to reduce the degree variability within each community. We simulated the spread of an epidemics in the network, given by the SIS dynamics of Eq.~\eqref{system SIS}. An important feature to study in such a system is whether the stable equilibrium is disease-free, which means that the disease will disappear completely in the long term, or if it is endemic, in which case the disease will remain. These two types of equilibria are characterized by $\langle \cal{X} \rangle = 0$ and $\langle \cal{X} \rangle > 0$, respectively. While the interactions are weak, the pathogen does not propagate enough through the network and the stable equilibrium is disease-free. The point at which an endemic stable equilibrium appears is almost perfectly captured by the 1-dimensional spectral reduction (Fig.~\ref{fig Maier_SIS}C), which again outperforms the degree-based reduction introduced by Gao et al.~\cite{gao_universal_2016} (Fig.~\ref{fig Maier_SIS}B). The homogeneous reduction is unable to capture this critical point even at $n=19$.

In the third example, we studied a neuronal dynamics on the \emph{C. elegans} connectome described in Ref.~\cite{chen_wiring_2006} where $N=279$. Nodes represent single neurons and the weighted and directed interactions represent connections among them (Fig.~\ref{fig CElegans_WC}A). We partitioned the nodes as we did for the contact network. The bifurcation diagram in this case is complex, with multiple bifurcation events, and it is not well represented by neither the degree-based reduction nor our reduced systems when the dimension is too small (Fig.~\ref{fig CElegans_WC}B~and~\ref{fig CElegans_WC}C). This can be interpreted as a signature of the system having a large effective dimension, much larger than the other two example systems. In general, for a fixed reduced dimension $n$, the spectral reduction provides a much more accurate picture of the bifurcation diagram than the homogeneous one.

\begin{figure}[ht]
  \centering
  \includegraphics[width=\linewidth ]{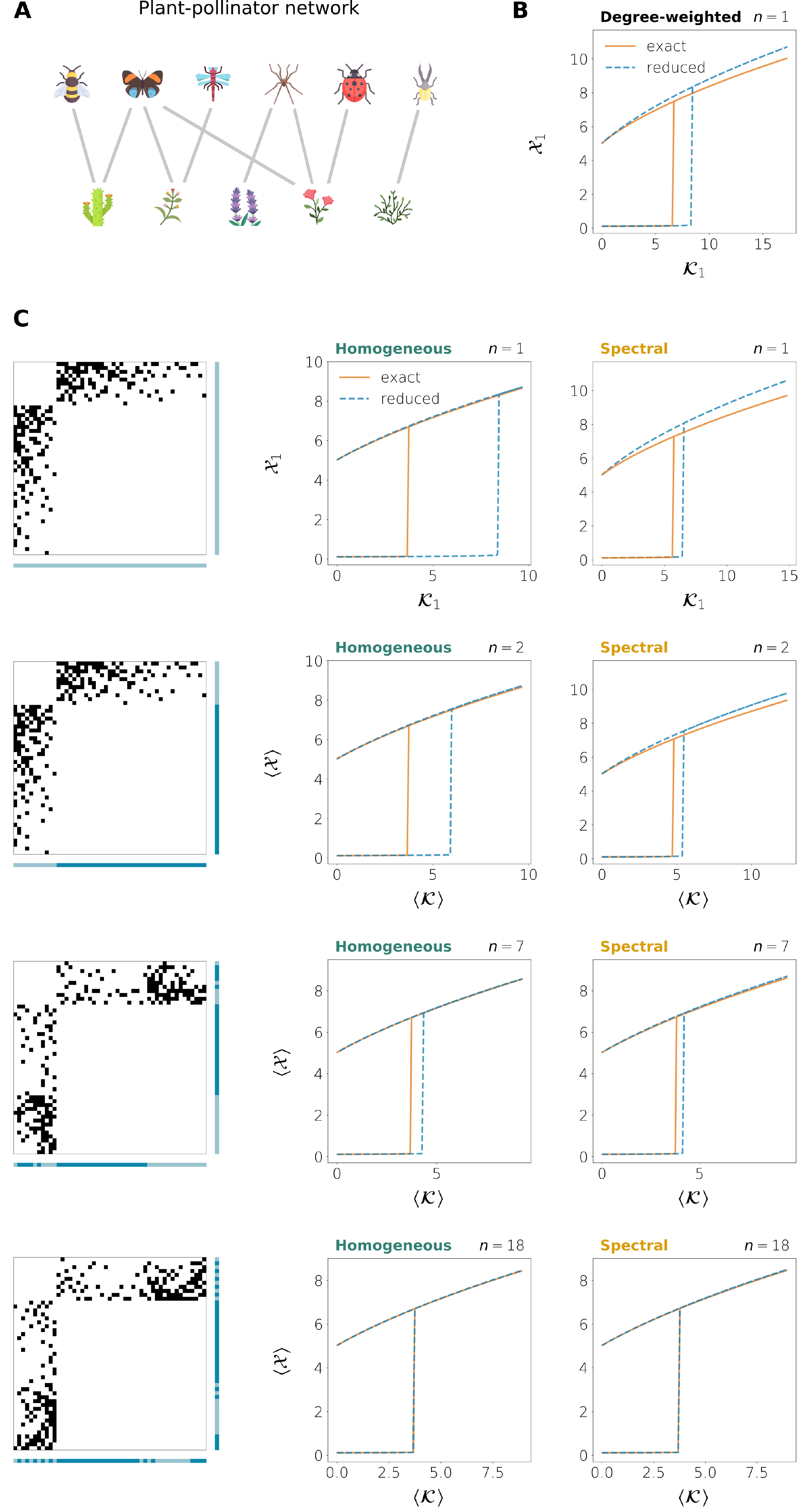}
  \caption{\small 
    Ecological dynamics (Eq.~\eqref{system ecology}), same parameters as in Fig.~\ref{fig hom n=2 eco}) on a plant-pollinator network of insects and plants in a Canary Island, Spain~\cite{dupont_structure_2003}. The interaction graph is bipartite, binary and undirected and contains 11 plants and 38 pollinators ($N=49$).
    \textbf{A}. Schematics of a plant-pollinator network. Icons have been taken from \href{https://flaticon.com}{\url{Flaticon.com}}.
    \textbf{B}. Bifurcation diagram obtained from the degree-based reduction defined in Ref.~\cite{gao_universal_2016}.
    \textbf{C}. Bifurcation diagrams for the homogeneous and the spectral methods when the whole network is taken as a single group ($n=1$) and for successive refinements of a partition on $n=2$ groups, each one representing one species type (plant or pollinator).
  }
  \label{fig Dupont_ecology}
\end{figure}

\begin{figure}[ht]
  \centering
  \includegraphics[width=\linewidth ]{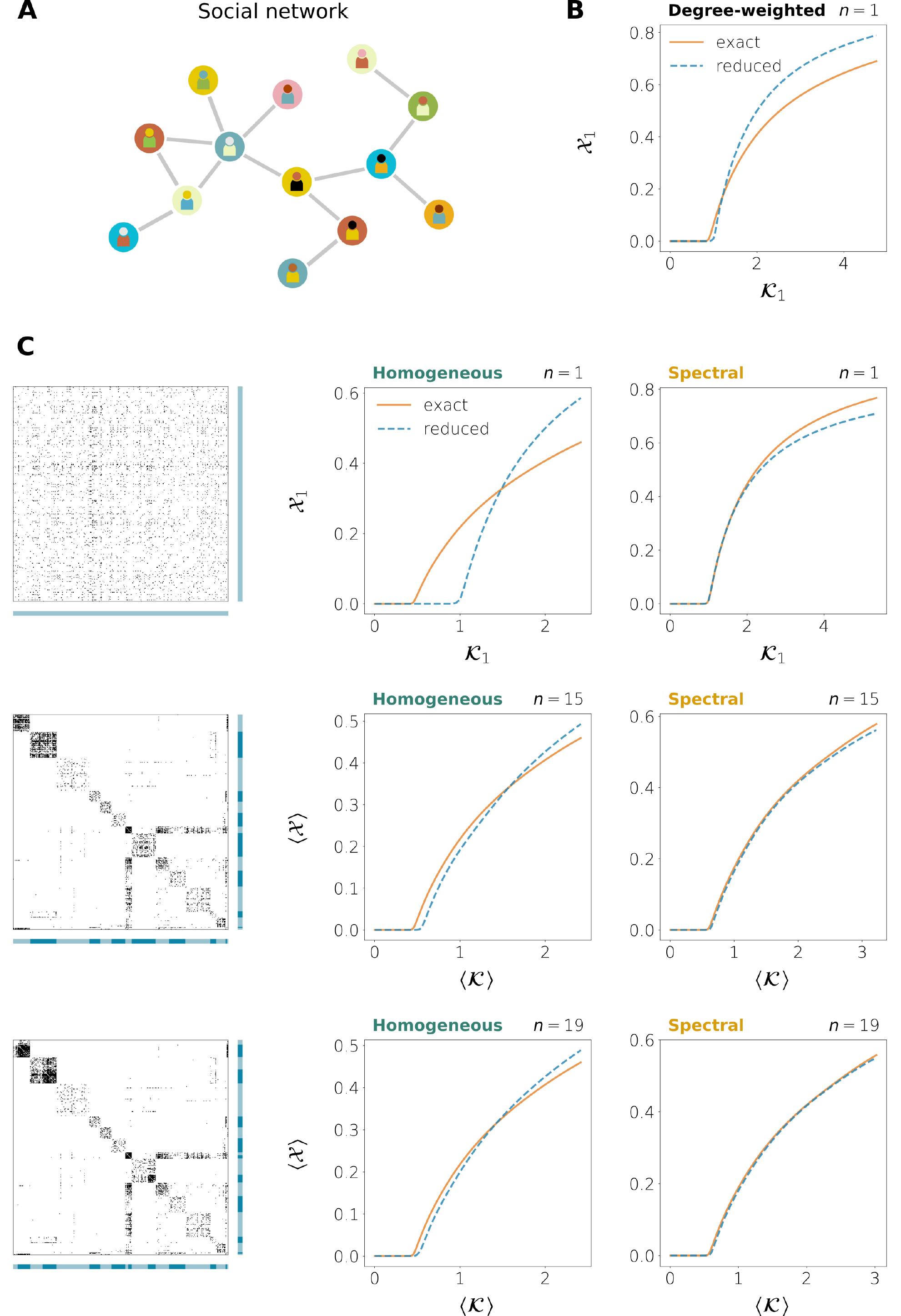}
  \caption{\small 
    SIS infectious dynamics (Eq.~\eqref{system SIS}, $\gamma=1$) on a social network based on Facebook contacts~\cite{maier_cover_2017}. The network has $N=362$ nodes and it is binary and undirected.
    \textbf{A}. Schematics of a social network.
    \textbf{B}. Bifurcation diagram obtained from the degree-based reduction defined in Ref.~\cite{gao_universal_2016}.
    \textbf{C}. Bifurcation diagrams for the homogeneous and the spectral methods when the whole network is taken as a single group ($n=1$) and for successive refinements of a partition on $n=15$ groups.
  }
  \label{fig Maier_SIS}
\end{figure}

\begin{figure}[ht]
  \centering
  \includegraphics[width=\linewidth ]{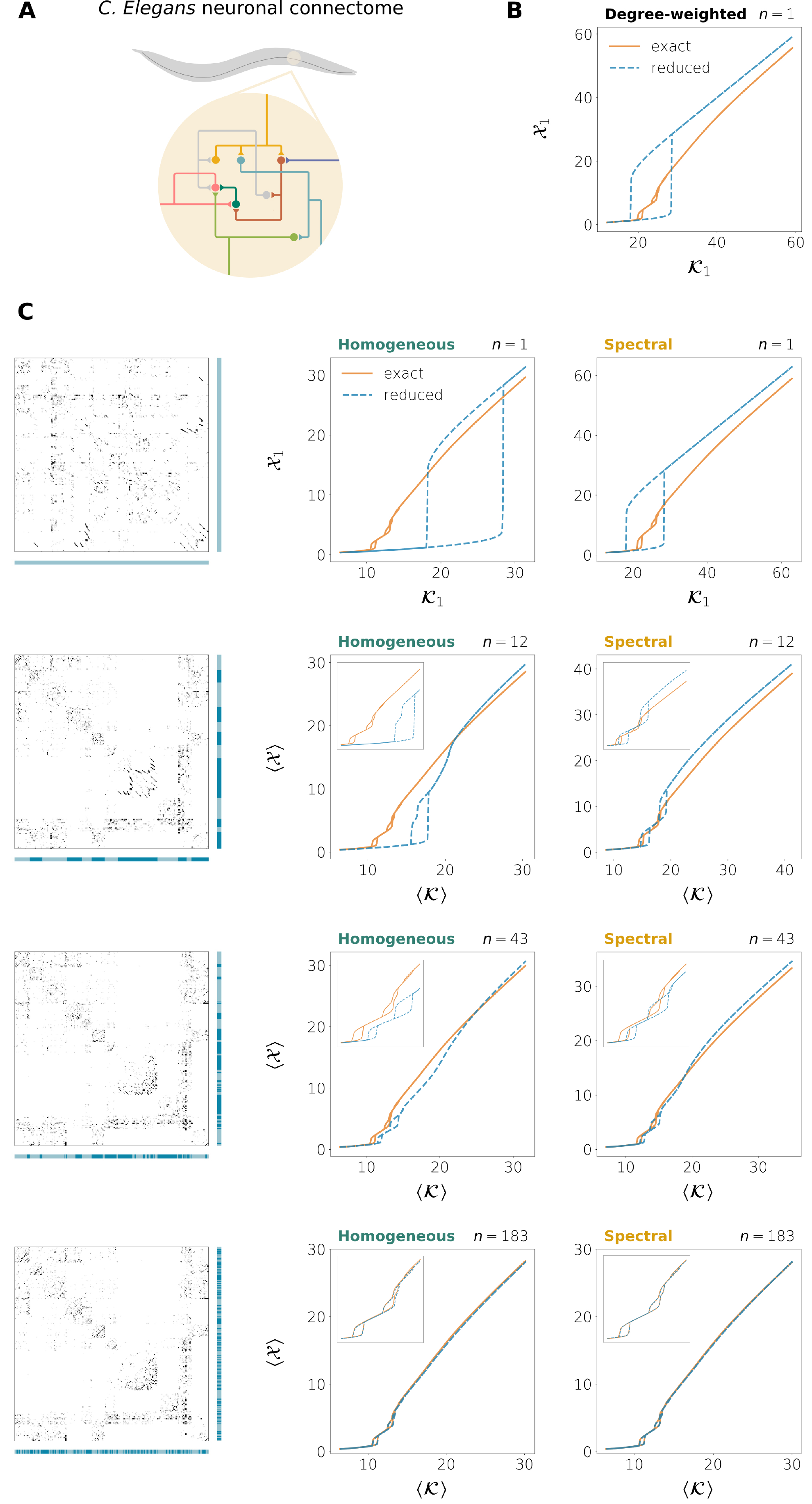}
  \caption{\small 
    Neuronal dynamics (Eq.~\eqref{system WC}), same parameters as in Fig.~\ref{fig hom wc}) on the connectome of the worm \emph{C. Elegans}~\cite{chen_wiring_2006}. The network is weighted and directed and contains $N=279$ neurons.
    \textbf{A}. Schematics of the \emph{C. Elegans} connectome.
    \textbf{B}. Bifurcation diagram obtained from the degree-based reduction defined in Ref. ~\cite{gao_universal_2016}.
    \textbf{C}. Bifurcation diagrams for the homogeneous and the spectral methods when the whole network is taken as a single group ($n=1$) and for successive refinements of a partition on $n=12$ groups. 
  }
  \label{fig CElegans_WC}
\end{figure}

\section{\label{sec:disc}Discussion}

We have presented a strategy to reduce the dimension of a dynamical system on a network of interactions.
The variables of the reduced system, the observables, are weighted averages of the activities within $n$ groups of nodes in the network. The node partition is defined \emph{a priori} based on the structure of the adjacency matrix and it is supposed to maximize the similarity of nodes that are in the same group. 
The key step in our reduction strategy is to calculate the reduction vectors that are used to construct the observables from the node activities. These vectors fully determine the reduced approximate dynamics, including a reduced adjacency matrix that specifies the magnitude of the coupling between observables.

We described two methods for computing the reduction vectors. In what we call the \emph{homogeneous reduction}, the observables are obtained from homogeneously averaging the activities within the different groups of nodes. The approximated reduced dynamics on these observables has the same form as the original dynamics. Also, the reduced adjacency matrix is such that the interaction from group $G_\rho$ to group $G_\nu$ is given by the average weighted in-degree that nodes in $G_\nu$ receive from nodes in $G_\rho$. This corresponds to what a naive observer would do to coarse-grain the original system.

Systems that are highly heterogeneous or for which it is difficult to define proper node partitions will not typically be well reduced by such a homogeneous coarse-graining. The main result of the present work is the definition of another procedure to construct the observables which can better cope with heterogeneities in the structure of interactions. In what we dub the \emph{spectral reduction}, the reduction vectors are no longer homogeneous over the nodes that form each of the different groups. Instead, they weigh the nodes differently so as to minimize the error of the approximated reduced dynamics. Finding the reduction vectors in this case requires solving a set of compatibility equations on these vectors. Despite the compatibility equations being generically incompatible, we proposed an algorithm for finding an approximate solution when the adjacency matrix is positive. The resulting approximate reduced dynamics is analogous in form to the original system except for the addition of a correction term.

We verified that both the homogeneous and the spectral reduction are suitable for reducing systems in which a proper node partition is identified and nodes in the same group have very similar connectivity profiles. In hindsight, this may come as no surprise. Indeed, a system composed of groups of nearly equivalent nodes can be naively reduced by identifying each group with its activity average; the averages then approximately obey the same dynamical laws as the original system with interaction strengths that come from averaging the weighted in-degrees from the different groups. This is exactly what our homogeneous reduction does. In these homogeneous networks with properly identified community structure, the spectral reduction provides almost no gain with respect to the homogeneous one.

We then analyzed the performance of the spectral reduction in more heterogeneous directed networks with known community structure. We found that the spectral reduction outperforms the homogeneous one and that it can be used to effectively reduce the dimension of these systems when a proper node partition is defined. The quality of the reduction can be enhanced by refining the partition so as to reduce the heterogeneity of nodes within each group. Doing so of course increases the reduced dimension $n$ and, with it, the complexity of the reduced dynamics. Our results suggest, however, that it is possible to find a compromise between increasing the number of groups and obtaining a reduced system of dimension significantly smaller than the original one because, unlike the homogeneous reduction, the spectral reduction can cope with a certain degree of inter-group heterogeneity. Similar results are obtained when real networks are analyzed, even if we do not know what the true community structure, if any, is. Provided that a prior study of their structure is performed to identify possible communities, the spectral reduction could contribute to understand and identify relevant features in the behavior of real complex systems.

However, classifying nodes into groups of similar connectivity profiles can be a difficult task. Even if different algorithms to this end are currently available, a search for communities will generally be imperfect, either because there are aspects of network architecture that can be difficult to extract or because a modular structure does not really exist in such networks. The presence of communities is often just an abstract notion that helps us to dissect and understand networks, but that fits to a network's structure only vaguely. For this reason, we would like our reduction method to be not too sensitive to errors in the definition of the node partition. We discovered that the spectral reduction is quite robust to partition perturbations, which makes it a good candidate to reduce systems in which communities cannot be well defined or identified.

Compared to other work on this topic, the dimension reduction strategy we have presented offers some important advantages. 
First, it can be applied to systems in which the adjacency matrix is weighted and not necessarily symmetric (i.e. directed networks). 
Second, the dimension of the reduced dynamics, $n$, is a variable that can be chosen to better adapt the reduction to the particular system under study. This feature adds flexibility to the reduction and can be very useful when dealing with systems of different degrees of complexity. The dimension of a proper reduction also provides information about the effective dimension of the system under study, and this information would be more difficult to infer if the dimension of the reduction were fixed a priori.
Finally, since our observables are weighted averages of the node activities within the different groups, they have a very clear interpretation. The reduced system not only allows us to approximate critical parameters at which the system's behavior qualitatively changes, but provides us with a direct measure of how the activity within these groups will be affected by a parameter perturbation around the critical value. For example, in a system modeling an ecological dynamics, our observables reflect the overall abundance of different groups of species, so by studying the reduced system we can infer which species groups could go extinct as a parameter is perturbed. 

Our work leaves also some open questions. To construct the spectral reduction, one has to approximately solve the compatibility equations on the reduction vectors. We have proposed a method for doing so when the adjacency matrix is positive, but other procedures should be defined to deal with general types of interactions. In neuronal networks, for example, inhibition plays a crucial role along with excitation, but so far our method can only be applied to networks composed exclusively of excitatory or inhibitory units.
Another limitation of the present work is that it is still unclear how to find the dimension $n$ of a proper reduction without comparing the reduced system with the original one. Since the performance of the reduction is tightly related to the properties of the chosen partition---i.e. the number of groups and the intra-group heterogeneity---it is likely that some measure on the network's structure under the chosen partition can be used to predict how accurate the reduced dynamics will be. Elucidating these and other issues will require additional research in the coming years.

\section*{Acknowledgments}
 This work was supported by the Fonds de recherche du Qu\'ebec -- Nature et technologies (V.T.), the Natural Sciences and Engineering Research Council of Canada (V.T., P.D., A.A.), and the Sentinel North program of Universit\'e Laval, funded by the Canada First Research Excellence Fund (M.V., V.T., P.D., A.A.). Moreover, we would like to thank Edward Laurence for stimulating discussions; his approximation method by partitioning~\cite[Annexe B.1.2]{laurence2020thesis} has influenced the early stages of our research. We also acknowledge Calcul Québec and Compute Canada for their technical support and computing infrastructures.

\bibliographystyle{unsrtnat}

\begin{thebibliography}{49}
\providecommand{\natexlab}[1]{#1}
\providecommand{\url}[1]{\texttt{#1}}
\expandafter\ifx\csname urlstyle\endcsname\relax
  \providecommand{\doi}[1]{doi: #1}\else
  \providecommand{\doi}{doi: \begingroup \urlstyle{rm}\Url}\fi

\bibitem[Donoho(2000)]{Donoho2000}
D.~L. Donoho.
\newblock High-dimensional data analysis: The curses and blessings of
  dimensionality.
\newblock \emph{AMS Conf. Math Challenges of the 21st Century}, 2000.

\bibitem[Aggarwal and Reddy(2014)]{Aggarwal2014}
Charu~C. Aggarwal and Chandan~K. Reddy, editors.
\newblock \emph{Data clustering: algorithms and applications}.
\newblock {Chapman and Hall/CRC}, 2014.
\newblock URL
  \url{https://www.routledge.com/Data-Clustering-Algorithms-and-Applications/Aggarwal-Reddy/p/book/9781466558212}.

\bibitem[Brunton and Kutz(2019)]{Brunton2019}
S.~L. Brunton and J.~N. Kutz.
\newblock \emph{{Data-Driven Science and Engineering: Machine Learning,
  Dynamical Systems, and Control}}.
\newblock Cambridge University Press, 2019.

\bibitem[Antoulas(2005)]{Antoulas2005}
A.~C. Antoulas.
\newblock \emph{{Approximation of Large-Scale Dynamical System}}.
\newblock SIAM, 2005.

\bibitem[Schilders et~al.(2008)Schilders, van~der Vorst, and
  Rommes]{Schilders2008}
W.~H.~A. Schilders, H.~A. van~der Vorst, and J.~Rommes, editors.
\newblock \emph{{Model Order Reduction: Theory, Research Aspects and
  Applications}}.
\newblock Springer, 2008.

\bibitem[Wei and Kuo(1969)]{Wei1969}
J.~Wei and J.~C.~W. Kuo.
\newblock {A lumping analysis in monomolecular reaction systems: Analysis of
  the exactly lumpable system}.
\newblock \emph{Ind. Eng. Chem. Fundamen.}, 8:\penalty0 114, 1969.
\newblock \doi{10.1021/i160029a019}.
\newblock URL \url{https://doi.org/10.1021/i160029a019}.

\bibitem[Kuo and Wei(1969)]{Kuo1969}
J.~C.~W. Kuo and J.~Wei.
\newblock {A lumping analysis in monomolecular reaction systems: Analysis of
  the approximately lumpable system}.
\newblock \emph{Ind. Eng. Chem. Fundamen.}, 8:\penalty0 124, 1969.
\newblock URL \url{https://doi.org/10.1021/i160029a020}.

\bibitem[Li and Rabitz(1990)]{Li1990}
G.~Li and H.~Rabitz.
\newblock {A general analysis of approximate lumping in chemical kinetics}.
\newblock \emph{Chem. Eng. Sci.}, 45:\penalty0 977, 1990.
\newblock \doi{10.1016/0009-2509(90)85020-E}.
\newblock URL \url{https://doi.org/10.1016/0009-2509(90)85020-E}.

\bibitem[Mitchell(2009)]{mitchell2009complexity}
M.~Mitchell.
\newblock \emph{Complexity: A guided tour}.
\newblock Oxford University Press, 2009.

\bibitem[Thurner et~al.(2018)Thurner, Hanel, and
  Klimek]{thurner2018introduction}
S.~Thurner, R.~Hanel, and P.~Klimek.
\newblock \emph{Introduction to the theory of complex systems}.
\newblock Oxford University Press, 2018.

\bibitem[Ladyman and Wiesner(2020)]{ladyman2020complex}
J.~Ladyman and K.~Wiesner.
\newblock \emph{What is a complex system?}
\newblock Yale University Press, 2020.

\bibitem[Cheng and Scherpen(2021)]{cheng2021model}
X.~Cheng and J.M.A. Scherpen.
\newblock Model reduction methods for complex network systems.
\newblock \emph{Annu.~Rev.~Control Robot.~Auton.~Syst.}, 4:\penalty0 425--453,
  2021.

\bibitem[Boccaletti et~al.(2006)Boccaletti, Latora, Moreno, and
  Hwang]{boccaletti2006complex}
S.~Boccaletti, V.~Latora, M.~Moreno, Y.and~Chavez, and D.-U. Hwang.
\newblock Complex networks: Structure and dynamics.
\newblock \emph{Phys.~Rep.}, 424:\penalty0 175, 2006.
\newblock \doi{10.1016/j.physrep.2005.10.009}.

\bibitem[Barzel and Barab{\'a}si(2013)]{barzel2013universality}
B.~Barzel and A.-L. Barab{\'a}si.
\newblock Universality in network dynamics.
\newblock \emph{Nature physics}, 9:\penalty0 673, 2013.
\newblock \doi{10.1038/nphys2741}.

\bibitem[Gerstner and Kistler(2002)]{gerstner2002mathematical}
W.~Gerstner and W.~M. Kistler.
\newblock Mathematical formulations of hebbian learning.
\newblock \emph{Biol.~Cybern}, 87:\penalty0 404, 2002.
\newblock \doi{10.1007/s00422-002-0353-y}.

\bibitem[Zenke and Gerstner(2017)]{zenke2017hebbian}
Friedemann Zenke and Wulfram Gerstner.
\newblock Hebbian plasticity requires compensatory processes on multiple
  timescales.
\newblock \emph{Philos.~Trans.~R.~Soc.~B: Biol.~Sci.}, 372:\penalty0 20160259,
  2017.
\newblock \doi{10.1098/rstb.2016.0259}.

\bibitem[Gao et~al.(2016)Gao, Barzel, and Barab{\'a}si]{gao_universal_2016}
J.~Gao, B.~Barzel, and A.-L. Barab{\'a}si.
\newblock Universal resilience patterns in complex networks.
\newblock \emph{Nature}, 530:\penalty0 307--312, 2016.
\newblock \doi{10.1038/nature16948}.

\bibitem[Jiang et~al.(2018)Jiang, Huang, Seager, Lin, Grebogi, Hastings, and
  Lai]{jiang_predicting_2018}
J.~Jiang, Z.-G. Huang, T.~P. Seager, W.~Lin, C.~Grebogi, A.~Hastings, and Y.-C.
  Lai.
\newblock Predicting tipping points in mutualistic networks through dimension
  reduction.
\newblock \emph{Proc. Natl. Acad. Sci. U.S.A.}, 115:\penalty0 E639--E647, 2018.
\newblock \doi{10.1073/pnas.1714958115}.

\bibitem[Tu et~al.(2021)Tu, D'Odorico, and Suweis]{tu_dimensionality_2021}
C.~Tu, P.~D'Odorico, and S.~Suweis.
\newblock Dimensionality reduction of complex dynamical systems.
\newblock \emph{iScience}, 24:\penalty0 101912, 2021.
\newblock \doi{10.1016/j.isci.2020.101912}.

\bibitem[Broido and Clauset(2019)]{broido2019scale}
A.~D. Broido and A.~Clauset.
\newblock Scale-free networks are rare.
\newblock \emph{Nat. Commun.}, 10:\penalty0 1--10, 2019.
\newblock \doi{10.1038/s41467-019-08746-5}.

\bibitem[Girvan and Newman(2002)]{girvan2002community}
M.~Girvan and M.E.J. Newman.
\newblock Community structure in social and biological networks.
\newblock \emph{Proc. Natl. Acad. Sci. U.S.A.}, 99:\penalty0 7821--7826, 2002.
\newblock \doi{10.1073/pnas.122653799}.

\bibitem[Sporns and Betzel(2016)]{sporns_modular_2016}
O.~Sporns and R.~F. Betzel.
\newblock Modular {{Brain Networks}}.
\newblock \emph{Annu. Rev. Psychol.}, 67:\penalty0 613--640, 2016.
\newblock \doi{10.1146/annurev-psych-122414-033634}.

\bibitem[Laurence et~al.(2019)Laurence, Doyon, Dub{\'e}, and
  Desrosiers]{laurence_spectral_2019}
E.~Laurence, N.~Doyon, L.~J. Dub{\'e}, and P.~Desrosiers.
\newblock Spectral {{Dimension Reduction}} of {{Complex Dynamical Networks}}.
\newblock \emph{Phys. Rev. X}, 9:\penalty0 011042, 2019.
\newblock \doi{10.1103/PhysRevX.9.011042}.

\bibitem[Thibeault et~al.(2020)Thibeault, {St-Onge}, Dub{\'e}, and
  Desrosiers]{thibeault_threefold_2020}
V.~Thibeault, G.~{St-Onge}, L.~J. Dub{\'e}, and P.~Desrosiers.
\newblock Threefold way to the dimension reduction of dynamics on networks:
  {{An}} application to synchronization.
\newblock \emph{Phys. Rev. Research}, 2:\penalty0 043215, 2020.
\newblock \doi{10.1103/PhysRevResearch.2.043215}.

\bibitem[Ott and Antonsen(2008)]{Ott2008a}
E.~Ott and T.~M. Antonsen.
\newblock {Low dimensional behavior of large systems of globally coupled
  oscillators}.
\newblock \emph{Chaos}, 18:\penalty0 037113, 2008.
\newblock \doi{10.1063/1.2930766}.

\bibitem[Pikovsky and Rosenblum(2008)]{Pikovsky2008}
A.~Pikovsky and M.~Rosenblum.
\newblock {Partially integrable dynamics of hierarchical populations of coupled
  oscillators}.
\newblock \emph{Phys. Rev. Lett.}, 101:\penalty0 264103, 2008.
\newblock \doi{10.1103/PhysRevLett.101.264103}.

\bibitem[Gfeller and De~Los~Rios(2008)]{Gfeller2008}
D.~Gfeller and P.~De~Los~Rios.
\newblock Spectral coarse graining and synchronization in oscillator networks.
\newblock \emph{Phys. Rev. Lett.}, 100:\penalty0 174104, 2008.
\newblock \doi{10.1103/PhysRevLett.100.174104}.

\bibitem[Hancock and Gottwald(2018)]{Hancock2018a}
E.~J. Hancock and G.~A. Gottwald.
\newblock {Model reduction for Kuramoto models with complex topologies}.
\newblock \emph{Phys. Rev. E}, 98:\penalty0 012307, 2018.
\newblock \doi{10.1103/PhysRevE.98.012307}.

\bibitem[Snyder et~al.(2020)Snyder, Zlotnik, and Lokhov]{Snyder2020a}
J.~Snyder, A.~Zlotnik, and A.~Y. Lokhov.
\newblock {Data-driven selection of coarse-grained models of coupled
  oscillators}.
\newblock \emph{Phys. Rev. Res.}, 2:\penalty0 043402, 2020.
\newblock \doi{10.1103/physrevresearch.2.043402}.

\bibitem[Hartwell et~al.(1999)Hartwell, Hopfield, Leibler, and
  Murray]{hartwell_molecular_1999}
L.~H. Hartwell, J.~J. Hopfield, S.~Leibler, and A.~W. Murray.
\newblock From molecular to modular cell biology.
\newblock \emph{Nature}, 402:\penalty0 C47--C52, 1999.
\newblock \doi{10.1038/35011540}.

\bibitem[Newman(2006)]{newman_modularity_2006}
M.~E.~J. Newman.
\newblock Modularity and community structure in networks.
\newblock \emph{Proc. Natl. Acad. Sci. U.S.A.}, 103:\penalty0 8577--8582, 2006.
\newblock \doi{10.1073/pnas.0601602103}.

\bibitem[Ravasz et~al.(2002)Ravasz, Somera, Mongru, Oltvai, and
  Barab{\'a}si]{ravasz_hierarchical_2002}
E.~Ravasz, A.~L. Somera, D.~A. Mongru, Z.~N. Oltvai, and A.-L. Barab{\'a}si.
\newblock Hierarchical {{Organization}} of {{Modularity}} in {{Metabolic
  Networks}}.
\newblock \emph{Science}, 297:\penalty0 1551--1555, 2002.
\newblock \doi{10.1126/science.1073374}.

\bibitem[Thibeault(2020)]{Thibeault2020_master}
V.~Thibeault.
\newblock R\'eduire la dimension des syst\`emes complexes~: un regard sur
  l'\'emergence de la synchronisation.
\newblock {{MSc Thesis}}, Universit\'e Laval, 2020.
\newblock URL \url{http://hdl.handle.net/20.500.11794/67313}.

\bibitem[Doreian et~al.(2019)Doreian, Batagelj, and Ferligoj]{Doreian2020}
P.~Doreian, V.~Batagelj, and A.~Ferligoj, editors.
\newblock \emph{Advances in {{Network Clustering}} and {{Blockmodeling}}}.
\newblock {Wiley}, 2019.
\newblock \doi{10.1002/9781119483298}.

\bibitem[Fortunato and Hric(2016)]{Fortunato2016}
S.~Fortunato and D.~Hric.
\newblock Community detection in networks: {{A}} user guide.
\newblock \emph{Phys. Rep.}, 659:\penalty0 1--44, 2016.
\newblock \doi{10.1016/j.physrep.2016.09.002}.

\bibitem[Peixoto(2021)]{peixoto2021descriptive}
T.~P. Peixoto.
\newblock Descriptive vs. inferential community detection: pitfalls, myths and
  half-truths.
\newblock 2021.
\newblock URL \url{http://arxiv.org/abs/2112.00183}.

\bibitem[Hopfield(1984)]{hopfield_neurons_1984}
J.~J. Hopfield.
\newblock Neurons with graded response have collective computational properties
  like those of two-state neurons.
\newblock \emph{Proc. Natl. Acad. Sci. U.S.A.}, 81:\penalty0 3088--3092, 1984.
\newblock \doi{10.1073/pnas.81.10.3088}.

\bibitem[Grossberg(1988)]{grossberg1988nonlinear}
S.~Grossberg.
\newblock Nonlinear neural networks: Principles, mechanisms, and architectures.
\newblock \emph{Neural networks}, 1:\penalty0 17--61, 1988.
\newblock \doi{10.1016/0893-6080(88)90021-4}.

\bibitem[Vogels et~al.(2005)Vogels, Rajan, and Abbott]{vogels2005neural}
T.~P. Vogels, K.~Rajan, and L.~F. Abbott.
\newblock Neural network dynamics.
\newblock \emph{Annu. Rev. Neurosci.}, 28:\penalty0 357--376, 2005.
\newblock \doi{10.1146/annurev.neuro.28.061604.135637}.

\bibitem[Pastor-Satorras et~al.(2015)Pastor-Satorras, Castellano, Van~Mieghem,
  and Vespignani]{pastor2015epidemic}
R.~Pastor-Satorras, C.~Castellano, P.~Van~Mieghem, and A.~Vespignani.
\newblock Epidemic processes in complex networks.
\newblock \emph{Rev. mod. phys.}, 87:\penalty0 925--979, 2015.
\newblock \doi{10.1103/RevModPhys.87.925}.

\bibitem[Holland et~al.(2002)Holland, DeAngelis, and
  Bronstein]{holland_population_2002}
J.~N. Holland, D.~L. DeAngelis, and J.~L. Bronstein.
\newblock Population {{Dynamics}} and {{Mutualism}}: {{Functional Responses}}
  of {{Benefits}} and {{Costs}}.
\newblock \emph{The American Naturalist}, 159:\penalty0 231--244, 2002.
\newblock \doi{10.1086/338510}.

\bibitem[Holland et~al.(1983)Holland, Blackmond~Laskey, and
  Leinhardt]{Holland1983}
P.~W. Holland, K.~Blackmond~Laskey, and S.~Leinhardt.
\newblock Stochastic blockmodels: {{First}} steps.
\newblock \emph{Soc. Networks}, 5:\penalty0 109--137, 1983.
\newblock \doi{10.1016/0378-8733(83)90021-7}.

\bibitem[Chung and Lu(2002{\natexlab{a}})]{Chung2002a}
F.~Chung and L.~Lu.
\newblock Connected {{Components}} in {{Random Graphs}} with {{Given Expected
  Degree Sequences}}.
\newblock \emph{Ann. Comb.}, 6:\penalty0 125--145, 2002{\natexlab{a}}.
\newblock \doi{10.1007/PL00012580}.

\bibitem[Chung and Lu(2002{\natexlab{b}})]{Chung2002pnas}
F.~Chung and L.~Lu.
\newblock The average distances in random graphs with given expected degrees.
\newblock \emph{Proc. Natl. Acad. Sci. U.S.A.}, 99:\penalty0 15879--15882,
  2002{\natexlab{b}}.
\newblock \doi{10.1073/pnas.252631999}.

\bibitem[Dupont et~al.(2003)Dupont, Hansen, and Olesen]{dupont_structure_2003}
Y.~L. Dupont, D.~M. Hansen, and J.~M. Olesen.
\newblock Structure of a plant\textendash flower-visitor network in the
  high-altitude sub-alpine desert of {{Tenerife}}, {{Canary Islands}}.
\newblock \emph{Ecography}, 26:\penalty0 301--310, 2003.
\newblock \doi{10.1034/j.1600-0587.2003.03443.x}.

\bibitem[Maier and Brockmann(2017)]{maier_cover_2017}
B.~F. Maier and D.~Brockmann.
\newblock Cover time for random walks on arbitrary complex networks.
\newblock \emph{Phys. Rev. E}, 96:\penalty0 042307, 2017.
\newblock \doi{10.1103/PhysRevE.96.042307}.

\bibitem[Peixoto(2014)]{peixoto_efficient_2014}
T.~P. Peixoto.
\newblock Efficient {{Monte Carlo}} and greedy heuristic for the inference of
  stochastic block models.
\newblock \emph{Phys. Rev. E}, 89:\penalty0 012804, 2014.
\newblock \doi{10.1103/PhysRevE.89.012804}.

\bibitem[Chen et~al.(2006)Chen, Hall, and Chklovskii]{chen_wiring_2006}
B.~L. Chen, D.~H. Hall, and D.~B. Chklovskii.
\newblock Wiring optimization can relate neuronal structure and function.
\newblock \emph{Proc. Natl. Acad. Sci. U.S.A.}, 103:\penalty0 4723--4728, 2006.
\newblock \doi{10.1073/pnas.0506806103}.

\bibitem[Laurence(2020)]{laurence2020thesis}
E.~Laurence.
\newblock \emph{La r\'esilience des r\'eseaux complexes}.
\newblock PhD thesis, Universit\'e Laval, 2020.
\newblock URL \url{http://hdl.handle.net/20.500.11794/66989}.

\end{thebibliography}

\pagebreak
\renewcommand{\theequation}{S\arabic{equation}}
\renewcommand{\thefigure}{S\arabic{figure}}
\setcounter{page}{1}
\setcounter{equation}{0}
\setcounter{figure}{0}
\onecolumn
\thispagestyle{plain}

\begin{center}
  {\LARGE Dimension reduction of dynamics on modular and heterogeneous directed networks}\vspace{0.25\baselineskip}\\
  {\large --- Supplementary Information ---}
\end{center}

\section{Preliminaries}

We start by presenting some preliminary assumptions and definitions.

\begin{assumption}[Complete system] \label{ass system}
For $N \in \mathbb{N}$, there are $N$ activity functions $x_1, \cdots, x_N$ of class $\mathscr{C}^2$ with
\begin{equation}
\begin{array}{cccc}
x_i: & \mathbb{R} & \rightarrow & \mathbb{R} \\
& t & \mapsto & x_i(t).
\end{array}
\end{equation}
These functions fulfill the system of ODEs
\begin{equation}
  \dot{x}_i = f(x_i) + \sum \limits_{j=1}^N w_{ij} \, g(x_i, x_j), \qquad i \in \{1, \cdots, N\},
  \label{system app}
\end{equation}
where $f: \, \mathbb{R} \rightarrow \mathbb{R}$, $g: \, \mathbb{R}^2 \rightarrow \mathbb{R}$ are functions of class $\mathscr{C}^1$, and $w_{ij}$ denotes a real number for all $i,j\in \{1,\cdots,N\}$.
\end{assumption}

\begin{definition}[Adjacency and in-degree matrices]
Given a system under Assumption \ref{ass system}, we call $\bm{W}=(w_{ij})_{i,j=1}^N$ the \emph{adjacency matrix}. The latter defines  a weighted directed network of $N$ nodes and $M$ links, where $M$ is the number of nonzero elements in $\bm{W}$. We define the \emph{in-degree matrix} as the $N \times N$ diagonal matrix constructed from the weighted in-degrees of the nodes in the network, that is,
\begin{equation}
\bm{K} = \text{diag}(k_1, \cdots, k_N), \hspace{0,3cm}
k_i = \sum \limits_{j=1}^N w_{ij}.
\end{equation}
\end{definition}

\begin{assumption}[Reduced system] \label{ass reduction}
Given a system defined by Assumption \ref{ass system}, there is a $n \in \mathbb{N}$, $n \leq N$, and a partition of the nodes into $n$ non-empty groups $G_1, \cdots, G_n$, that is, $G_\nu \neq \varnothing$ for all $\nu$, $\cup_{\nu=1}^n G_\nu = \{1,\cdots,N\}$, and $G_\nu \cap G_\rho = \varnothing$ for $\nu \neq \rho$. There is also a set of $n$ reduction vectors $\bm{a}_1, \cdots, \bm{a}_n$ where $\bm{a}_\nu = (a_{\nu i})_{i=1}^N \in \mathbb{R}^N$ is the reduction vector associated to group $G_\nu$, and
\begin{equation}
\begin{array}{lllll}
& \displaystyle \sum \limits_{i=1}^N a_{\nu i} = 1,
& a_{\nu i} = 0 \,\,\text{ if }\,\, i \notin G_\nu ,
& \nu \in \{1, \cdots, n\}.
\end{array}
\end{equation}
There are $n$ observables $\cal{X}_1, \cdots, \cal{X}_n$ of class $\mathscr{C}^1$ with
\begin{equation}
\begin{array}{cccc}
{\cal X}_\nu: & \mathbb{R} & \rightarrow & \mathbb{R} \\
& t & \mapsto & {\cal X}_\nu(t)
\end{array}
\end{equation}
which are constructed from the reduction vectors and the activity functions through
\begin{equation}
\cal{X}_\nu := \sum \limits_{i=1}^N a_{\nu i} x_i, \hspace{0,3cm} \nu \in \{1, \cdots, n\}.
\label{observables app}
\end{equation}
\end{assumption}

\begin{assumption}[Ordered partition] \label{ass order}
Let $\{G_1, \cdots, G_n\}$ be a partition of $\{1, \cdots, N\}$, that is, $G_\nu \neq \varnothing$ for all $\nu$, $\cup_{\nu=1}^n G_\nu = \{1,\cdots,N\}$, and $G_\nu \cap G_\rho = \varnothing$ for $\nu \neq \rho$. The indices within each partition set are consecutive integers:
\begin{equation}
G_\nu = \left\lbrace 1 + \sum \limits_{\rho=1}^{\nu-1} m_\rho, \cdots, \sum \limits_{\rho=1}^{\nu} m_\rho \right\rbrace,
\end{equation}
where $m_\nu = | G_\nu |$ is the size of $G_\nu$.
Each index in $\{1, \cdots, m_\nu \}$ is mapped to an index in $G_\nu \subseteq \{1, \cdots, N\}$ by
\begin{equation}
\begin{array}{rrll}
p_\nu: & \{1, \cdots, m_\nu \} & \rightarrow & G_\nu \\
& i & \mapsto &
p_\nu(i) := i + \sum \limits_{\rho=1}^{\nu-1} m_\rho.
\end{array}
\end{equation}

\end{assumption}

Without loss of generality we can always suppose that Assumption \ref{ass order} holds when dealing with a system defined by Assumptions \ref{ass system} and \ref{ass reduction} (it is enough to permute the node indices $1, \cdots, N$). This allows us to express the adjacency and in-degree matrices as follows.

\begin{definition}[Group-to-group interaction and in-degree matrices] \label{def submatrices}
Under Assumptions \ref{ass system} and \ref{ass order}, for each pair of indices $\nu, \rho \in \{1, \cdots, n\}$, we define the submatrix of interactions from $G_\rho$ to $G_\nu$, $\bm{W}_{\nu \rho} = ( [\bm{W}_{\nu \rho}]_{ij})_{i,j}$, of dimension $m_\nu \times m_\rho$, and the in-degree submatrix of nodes in $G_\nu$ from nodes in $G_\rho$, $\bm{K}_{\nu \rho} = \text{diag} ( [\bm{K}_{\nu \rho}]_{ii})_i$, of dimension $m_\nu \times m_\nu$, as
\begin{equation}
\begin{array}{lllll}
[\bm{W}_{\nu \rho}]_{ij} &=& w_{p_\nu(i) p_\rho(j)} \\\\
\left[ \bm{K}_{\nu \rho} \right]_{i i} &=& \displaystyle \sum \limits_{j \in G_\rho} w_{p_\nu(i) j}
&=& k_{p_\nu(i)}^\rho
\end{array}
\end{equation}
for $i \in \{1, \cdots, m_\nu \}$ and $j \in  \{1, \cdots, m_\rho \}$, where $ k_{p_\nu(i)}^\rho$ is the weighted in-degree of node $p_\nu(i)$ from $G_\rho$.
This allows to express the adjacency matrix $\bm{W}$ in the block form
\begin{equation}
\bm{W} = \left( 
\begin{array}{ccc}
\bm{W}_{11} & \cdots & \bm{W}_{1n} \\
\vdots & \ddots & \vdots \\
\bm{W}_{n1} & \cdots & \bm{W}_{nn}
\end{array} \right)
\label{decomposed adjacency matrix app}
\end{equation}
and the in-degree matrix as
\begin{equation}
\bm{K} = \left( 
\begin{array}{ccc}
\bm{K}_{11} + \cdots + \bm{K}_{1n} & \cdots & 0 \\
\vdots & \ddots & \vdots \\
0 & \cdots & \bm{K}_{n1} + \cdots + \bm{K}_{nn}
\end{array} \right).
\label{decomposed degree matrix app}
\end{equation}
\end{definition}

It is also useful to define, for each $\nu \in \{1, \cdots, n\}$, the components of the reduction vector $\bm{a}_\nu$ that correspond to the indices within $G_\nu$ (the other components are zero):

\begin{definition}[Partial reduction vectors] \label{def red vectors hat}
Under Assumptions \ref{ass reduction} and \ref{ass order}, for each $\nu \in \{1,\cdots,n\}$, 
the $\nu$-th \emph{partial reduction vector} 
$\bm{\widehat{a}}_\nu = ( \widehat{a}_{\nu i} )_{i=1}^{m_\nu} \in \mathbb{R}^{m_\nu}$ is such that
\begin{equation}
\widehat{a}_{\nu i} := a_{\nu \, p_\nu(i)},\qquad i\in\{1,\ldots, m_\nu\}.
\label{def a hat app}
\end{equation}
Notice that the normalization condition still holds on $\bm{\widehat{a}}_\nu$:
$\sum \limits_{i=1}^{m_\nu} \widehat{a}_{\nu i} = 1.$ 
\end{definition}

\section{Approximate reduced dynamics} \label{sec: app closing the dynamics}

Here we describe two possible strategies for reducing the dimension of a system defined by Assumption \ref{ass system} by means of the observables presented in Assumption \ref{ass reduction}. In both cases, we approximate $f(x_i$) and $g(x_i,x_j)$ by Taylor polynomials around the observables associated to the groups to which $i$ and $j$ belong. The reason for this is that nodes will be partitioned into groups so that nodes in the same group have similar connectivity properties. If the partition is well chosen, the activities of nodes in the same group should be close to each other and also close to the corresponding observable. The two strategies result in what we have dubbed the \emph{homogeneous} and the \emph{spectral} reductions.
We first provide a lemma and some useful notation.

\begin{lemma}[Exact reduced dynamics]\label{lemma:Xnu dot app}
Under Assumptions \ref{ass system} and \ref{ass reduction}, the observables fulfill the system of ODEs
\begin{equation}
\cal{ \dot{X}}_\nu 
= \displaystyle \sum \limits_{i=1}^N a_{\nu i} f(x_i) + \sum \limits_{i,j=1}^N a_{\nu i} w_{ij} g(x_i,x_j),\qquad \nu\in\{1,\ldots, n\}.
\label{Xnu dot app}
\end{equation}
\end{lemma}

\begin{proof}
The result follows directly from Assumptions \ref{ass system} and \ref{ass reduction}.
\hfill\end{proof}

\begin{notation}[Big $\mathcal O$]
Let $g: \, \mathbb{R}^2 \rightarrow \mathbb{R}$, $h: \, \mathbb{R}^4 \rightarrow \mathbb{R}$ be two functions. Let $x, X \in \mathbb{R}$, 
$\bm{x} =(x_1,x_2)^T, \bm{X} =(X_1,X_2)^T \in \mathbb{R}^2$. 
We write
\begin{equation}
g(x, X) = {\cal O} \left( x - X \right)
\end{equation}
whenever there exists $K \in \mathbb{R}$ such that
\begin{equation}
\lim \limits_{x \rightarrow X} \frac{g(x,X)}{x-X} = K.
\end{equation}
We write
\begin{equation}
h(x_1, x_2, X_1, X_2) = {\cal O} \left( (x_1 - X_1) (x_2 - X_2) \right)
\end{equation}
whenever there exists $L \in \mathbb{R}$ such that
\begin{equation}
\lim \limits_{ \| \bm{x}-\bm{X} \| \rightarrow 0 } \frac{h(x_1,x_2,X_1,X_2)}{(x_1-X_1)(x_2-X_2)} = L.
\end{equation}
\end{notation}

\begin{notation}[Function approximation]
Let $f: \, \mathbb{R^N} \rightarrow \mathbb{R}$, $F: \, \mathbb{R}^n \rightarrow \mathbb{R}$ be two functions. Given $n \leq N$, let $\{G_1, \cdots, G_n\}$ be a partition of $\{1, \cdots, N\}$. For $\bm{x} \in \mathbb{R}^N$ and $\bm{X} \in \mathbb{R}^n$, we write
\begin{equation}
f( \bm{x} ) \overset{{\cal O}_1}{\approx} F( \bm{X} )
\end{equation}
whenever
\begin{equation}
f( \bm{x} ) - F( \bm{X} ) = \sum \limits_{\nu=1}^n \sum \limits_{i \in G_\nu} {\cal O} \left( x_i - X_\nu \right).
\end{equation}
We write
\begin{equation}
f( \bm{x} ) \overset{{\cal O}_2}{\approx} F( \bm{X} )
\end{equation}
whenever
\begin{equation}
f( \bm{x} ) - F( \bm{X} ) = \sum \limits_{\nu,\rho=1}^n \sum
\limits_{\substack{ i \in G_\nu \\ j \in G_\rho}}
{\cal O} \left( (x_i - X_\nu) (x_j - X_\rho) \right).
\end{equation}
\end{notation}

With this notation in hand, we can present the homogeneous and the spectral reductions.

\begin{proposition}[Homogeneous reduction] \label{prop homogeneous}
Suppose that Assumptions \ref{ass system} and \ref{ass reduction} hold.
If
\begin{equation}
a_{\nu i} = \left\lbrace 
\begin{array}{ll}
1/|G_\nu| & \text{ if } i \in G_\nu, \\
0 & \text{ otherwise},
\end{array} \right.
\end{equation}
then
\begin{equation}
\begin{array}{lll}
\cal{\dot{X}}_\nu 
&\overset{{\cal O}_1}{\approx} & \displaystyle 
f(\cal{X}_\nu)
+ \sum \limits_{\rho=1}^n \cal{W}_{\nu \rho} \, g(\cal{X}_\nu, \cal{X}_\rho), \\
\end{array}
\label{Xnu dot homogeneous app}
\end{equation}
where
\begin{equation}
\cal{W}_{\nu \rho} := \sum \limits_{\substack{ i \in G_\nu \\ j \in G_\rho}} a_{\nu i} w_{ij}
= \sum \limits_{i \in G_\nu} a_{\nu i} k_i^\rho.
\end{equation}
\end{proposition}

\begin{proof}
By Taylor's Theorem, we can approximate $f(x_i$) and $g(x_i,x_j)$ around the corresponding observables as
\begin{equation}
\begin{array}{lllll}
f(x_i) & \overset{{\cal O}_1}{\approx} & f( \cal{X}_\nu) && \text{ for } i \in G_\nu \\
g(x_i,x_j) & \overset{{\cal O}_1}{\approx} & g(\cal{X}_\nu, \cal{X}_\rho) && \text{ for } i \in G_\nu, j \in G_\rho.
\end{array}
\end{equation}
Taking into account that, by construction, $a_{\nu k} = 0$ whenever $k \notin G_\nu$, from Lemma \ref{lemma:Xnu dot app} we get
\begin{equation}
\begin{array}{lll}
\cal{\dot{X}}_\nu 
& \overset{{\cal O}_1}{\approx} & \displaystyle 
f(\cal{X}_\nu) \sum \limits_{i \in G_\nu} a_{\nu i} 
+ \sum \limits_{\rho=1}^n g(\cal{X}_\nu, \cal{X}_\rho) \sum \limits_{\substack{ i \in G_\nu \\ j \in G_\rho}} a_{\nu i} w_{ij} \\

&=& \displaystyle
f(\cal{X}_\nu)
+ \sum \limits_{\rho=1}^n \cal{W}_{\nu \rho} \, g(\cal{X}_\nu, \cal{X}_\rho).\\
\end{array}
\end{equation}
\hfill\end{proof}

\begin{proposition}[Spectral reduction] \label{prop spectral}
Suppose that Assumptions \ref{ass system}, \ref{ass reduction} and \ref{ass order} hold.
Let $\{ \bm{W}_{\nu \rho} \}_{\nu, \rho}$, $\{ \bm{K}_{\nu \rho} \}_{\nu, \rho}$ and $\{ \bm{\widehat{a}}_\nu \}_\nu$ be the sets of matrices and vectors of Definitions \ref{def submatrices} and \ref{def red vectors hat}.
If there exist two matrices 
$\bm{\mu}=(\mu_{\nu \rho})_{\nu,\rho}$
and 
$\bm{\lambda}=(\lambda_{\nu \rho})_{\nu,\rho}$
of dimension $n \times n$ such that the partial reduction vectors fulfill the compatibility equations
\begin{subequations}
\begin{align}
    \bm{K}_{\nu \rho} \bm{\widehat{a}}_\nu = \mu_{\nu \rho} \bm{\widehat{a}}_\nu,
    \label{comp eq K app} \\
    \bm{W}_{\nu \rho}^T \bm{\widehat{a}}_\nu = \lambda_{\nu \rho} \bm{\widehat{a}}_\rho,
    \label{comp eq W app}
\end{align}
\end{subequations}
then
\begin{equation}
\begin{array}{lll}
\cal{\dot{X}}_\nu 
& \overset{{\cal O}_2}{\approx} & \displaystyle 
f(\cal{X}_\nu)
+ \sum \limits_{\rho=1}^n \cal{W}_{\nu \rho} \, g(\cal{X}_\nu, \cal{X}_\rho), \\
\end{array}
\label{Xnu dot spectral without correction app}
\end{equation}
where
\begin{equation}
\cal{W}_{\nu \rho} := \sum \limits_{\substack{ i \in G_\nu \\ j \in G_\rho}} a_{\nu i} w_{ij}
= \sum \limits_{i \in G_\nu} a_{\nu i} k_i^\rho.
\label{def calW nu rho app}
\end{equation}

\end{proposition}

\begin{proof}
By Taylor's Theorem we can approximate $f(x_i$) and $g(x_i,x_j)$ at first order around the corresponding observables as
\begin{equation}
\begin{array}{lllll}
f(x_i) & \overset{{\cal O}_2}{\approx} & f(\cal{X}_\nu) + f'(\cal{X}_\nu) (x_i - \cal{X}_\nu) && \text{ for } i \in G_\nu, \\
g(x_i,x_j) & \overset{{\cal O}_2}{\approx} & g(\cal{X}_\nu, \cal{X}_\rho) + g_1(\cal{X}_\nu, \cal{X}_\rho) (x_i - \cal{X}_\nu) + g_2(\cal{X}_\nu, \cal{X}_\rho) (x_j - \cal{X}_\rho) 
&& \text{ for } i \in G_\nu, j \in G_\rho.
\end{array}
\label{approx f, g}
\end{equation}
We rewrite Eq.~\eqref{Xnu dot app} as
\begin{equation}
\begin{array}{lll}
\cal{\dot{X}}_\nu 
&=& \displaystyle \sum \limits_{i \in G_\nu} a_{\nu i} f(x_i) + \sum \limits_{\rho=1}^n \sum \limits_{\substack{ i \in G_\nu \\  j \in G_\rho}} a_{\nu i} w_{ij} g(x_i,x_j) \\
&=& \displaystyle T_\nu + \sum \limits_{\rho = 1}^n T_{\nu \rho}, 
\end{array}
\label{Xnu dot membership}
\end{equation}
where
\begin{equation}
\begin{array}{lll}
T_{\nu} &:=& \displaystyle \sum \limits_{i \in G_\nu} a_{\nu i} f(x_i), \\
T_{\nu \rho} &:=& \displaystyle \sum \limits_{\substack{ i \in G_\nu \\ j \in G_\rho}} a_{\nu i} w_{ij} g(x_i,x_j). \\
\end{array}
\label{Xnu dot sep}
\end{equation}
Using approximation~\eqref{approx f, g}, the definition of observables and the normalization condition on $\bm{a}_\nu$, $\sum \limits_{i \in G_\nu} a_{\nu i} = 1$, we obtain the following first-order approximations:
\begin{subequations}
\begin{equation}
\begin{array}{lll}
T_{\nu} & \overset{{\cal O}_2}{\approx} & \displaystyle f(\cal{X}_\nu) + f'(\cal{X}_\nu) \sum \limits_{i \in G_\nu} a_{\nu i} (x_i - \cal{X}_\nu) \\
&=& f(\cal{X}_\nu), 
\end{array}
\label{T nu}
\end{equation}
\begin{equation}
\begin{array}{lll}
T_{\nu \rho} & \overset{{\cal O}_2}{\approx} &
\displaystyle g(\cal{X}_\nu, \cal{X}_\rho) \sum \limits_{\substack{ i \in G_\nu \\ j \in G_\rho}} a_{\nu i} w_{ij} 
+  g_1(\cal{X}_\nu, \cal{X}_\rho) \sum \limits_{\substack{ i \in G_\nu \\ j \in G_\rho}} a_{\nu i} w_{ij} (x_i-\cal{X}_\nu)
+  g_2(\cal{X}_\nu, \cal{X}_\rho) \sum \limits_{\substack{ i \in G_\nu \\ j \in G_\rho}} a_{\nu i} w_{ij} (x_j-\cal{X}_{\rho}) \\

&=& \displaystyle \cal{W}_{\nu \rho} \, g(\cal{X}_\nu, \cal{X}_{\rho})
+  g_1(\cal{X}_\nu, \cal{X}_{\rho}) \left( \sum \limits_{\substack{ i \in G_\nu \\ j \in G_\rho}} a_{\nu i} w_{ij} x_i - \cal{W}_{\nu \rho} \cal{X}_\nu \right)
+  g_2(\cal{X}_\nu, \cal{X}_{\rho}) \left( \sum \limits_{\substack{ i \in G_\nu \\ j \in G_\rho}} a_{\nu i} w_{ij} x_j - \cal{W}_{\nu \rho} \cal{X}_{\rho} \right). \\
\end{array}
\label{T nu rho}
\end{equation}
\end{subequations}
We can rewrite the compatibility equations in component form as follows:
\begin{subequations}
\begin{align}
    \left[ \bm{K}_{\nu \rho} \right]_{ii} \widehat{a}_{\nu i}
    = \mu_{\nu \rho} \widehat{a}_{\nu i} 
    && i \in \{1, \cdots, m_\nu \},
    \label{comp eq K app 2} \\
   \sum \limits_{i=1}^{m_\nu} [ \bm{W}_{\nu \rho}]_{ij} \widehat{a}_{\nu i}
    = \lambda_{\nu \rho} \widehat{a}_{\rho j}
    && j \in \{1, \cdots, m_\rho \}.
    \label{comp eq W app 2}
\end{align}
\end{subequations}
This can in turn be expressed as a function of the general interaction and degree matrices and the complete reduction vectors as
\begin{subequations}
\begin{align}
    \sum \limits_{j \in G_\rho} w_{ij} a_{\nu i}
    = \mu_{\nu \rho} a_{\nu i} 
    & &  i \in G_\nu,
    \label{comp eq K app 3} \\
   \sum \limits_{i \in G_\nu} w_{ij} a_{\nu i}
    = \lambda_{\nu \rho} a_{\rho j}
    & &  j \in G_\rho.
    \label{comp eq W app 3}
\end{align}
\end{subequations}
Given an arbitrary activity vector $\bm{x}=(x_1, \cdots, x_N)^T$, we now multiply both sides of Eq.~\eqref{comp eq K app 3} by $x_i$ and we sum over $i \in G_\nu$. We also multiply both sides of Eq.~\eqref{comp eq W app 3} by $x_j$ and we sum over $j \in G_\rho$. We get 
\begin{subequations}
\begin{align}
    \sum \limits_{\substack{ i \in G_\nu \\ j \in G_\rho}} w_{ij} a_{\nu i} x_i
    = \mu_{\nu \rho} \sum \limits_{i \in G_\nu} a_{\nu i} x_i
    & = \mu_{\nu \rho} {\cal X}_\nu
    \label{comp eq K app 4} \\
   \sum \limits_{\substack{ i \in G_\nu \\ j \in G_\rho}} w_{ij} a_{\nu i} x_j
    = \lambda_{\nu \rho} \sum \limits_{j \in G_\rho} a_{\rho j} x_j
    & = \lambda_{\nu \rho} {\cal X}_\rho.
    \label{comp eq W app 4}
\end{align}
\end{subequations}
On the other hand, we can sum Eq.~\eqref{comp eq K app 3} over $i \in G_\nu$ and Eq.~\eqref{comp eq W app 3} over $j \in G_\rho$ to get
\begin{subequations}
\begin{align}
  {\cal W}_{\nu \rho} = &
    \sum \limits_{\substack{ i \in G_\nu \\ j \in G_\rho}} w_{ij} a_{\nu i}
    = \mu_{\nu \rho} \sum \limits_{i \in G_\nu} a_{\nu i} 
    = \mu_{\nu \rho} 
    \label{mu nu rho 2} \\
   {\cal W}_{\nu \rho} = &
   \sum \limits_{\substack{ i \in G_\nu \\ j \in G_\rho}} w_{ij} a_{\nu i}
    = \lambda_{\nu \rho} \sum \limits_{j \in G_\rho} a_{\rho j}
    = \lambda_{\nu \rho} .
    \label{lambda nu rho 2}
\end{align}
\end{subequations}
Plugging Eqs.~\eqref{comp eq K app 4},~\eqref{comp eq W app 4},~\eqref{mu nu rho 2} and~\eqref{lambda nu rho 2} into Eqs.~\eqref{T nu} and~\eqref{T nu rho} we get
\begin{subequations}
\begin{equation}
\begin{array}{lll}
T_{\nu} & \overset{{\cal O}_2}{\approx} f(\cal{X}_\nu) 
\end{array}
\end{equation}
\begin{equation}
\begin{array}{lll}
T_{\nu \rho} & \overset{{\cal O}_2}{\approx} &
\displaystyle \cal{W}_{\nu \rho} \, g(\cal{X}_\nu, \cal{X}_{\rho}),
\end{array}
\end{equation}
\end{subequations}
which completes the proof.
\hfill\end{proof}

Notice that, given a set of vectors $\bm{\widehat{a}}_1, \cdots, \bm{\widehat{a}}_n$, the matrices of scalars 
$\bm{\mu}=(\mu_{\nu \rho})_{\nu,\rho}$, 
$\bm{\lambda}=(\lambda_{\nu \rho})_{\nu,\rho}$
that solve or minimize the quadratic error in the compatibility equations~\eqref{comp eq K app},~\eqref{comp eq W app}
are given by
\begin{subequations}
\begin{align}
    \mu_{\nu \rho} = \frac{ \bm{\widehat{a}}_\nu^T \bm{K}_{\nu \rho} \bm{\widehat{a}}_\nu} {\| \bm{\widehat{a}}_\nu \|^2}
    \label{mu nu rho 1} \\
    \lambda_{\nu \rho} = \frac{ \bm{\widehat{a}}_\rho^T \bm{W}_{{\nu \rho}}^T \bm{\widehat{a}}_\nu} {\| \bm{\widehat{a}}_\rho \|^2}
    \label{lambda nu rho 1}
\end{align}
\end{subequations}
(see Lemma \ref{lemma min quadratic error} below).

\begin{lemma} \label{lemma min quadratic error}
Let $\bm{M}$ be a matrix of dimension $r \times s$ and let $\bm{u}$, $\bm{v} \neq \bm{0}$ be vectors of dimension $s$ and $r$, respectively. The scalar $\lambda$ that minimizes
\begin{equation}
\Vert \bm{M} \bm{u} - \lambda \bm{v} \Vert^2
\end{equation}
is given by
\begin{equation}
\lambda = {\bm v}^+\bm{M u} ,
\label{sol lambda}
\end{equation}
where
\begin{equation}\label{eq:pseudo-inverse-vec}
     {\bm v}^+ =\frac{1}{\|\bm v\|^2}\bm{v}^T
\end{equation}
is the Moore-Penrose pseudoinverse of the vector $\bm{v}$.
\end{lemma}

\begin{proof}
Define the function $F\,:\,\mathbb{R}\to\mathbb{R}$ as
\begin{equation}
F(\lambda) = \Vert \bm{M} \bm{u} - \lambda \bm{v} \Vert^2.
\end{equation}
This function is obviously smooth. Moreover, it is convex since $F''(\lambda) = 2 \langle \bm{v}, \bm{v} \rangle > 0$. The minimum of $F$ is thus found by solving $F'(\lambda) = 0$. However, 
\begin{equation}
F'(\lambda) 
= -2 \langle \bm{v}, \bm{M} \bm{u} - \lambda \bm{v} \rangle
= -2 \left( \bm{v}^T \bm{M} \bm{u} - \lambda \bm{v}^T \bm{v} \right),
\end{equation}
which is zero whenever $\bm{v}^T \bm{M} \bm{u} = \lambda \bm{v}^T \bm{v}$.
Therefore, the minimum of $F$ is defined by 
\begin{equation}
    \lambda = \frac{ \bm{v}^T \bm{ M u} }{\bm{v}^T \bm{v}},
\end{equation}
which in turn is equivalent to Eq.~\eqref{sol lambda} since ${\bm v}^+$, the Moore-Penrose pseudoinverse of $\bm{v}$, is equal to $\bm{v}^T /( \bm{v}^T\bm{v})= \bm{v}^T /\|\bm v\|^2 $. 
\hfill\end{proof}

\section{Correction when the compatibility equations cannot be solved exactly} \label{sec: app corrected reduced system}

The compatibility equations \eqref{comp eq K app}--\eqref{comp eq W app} provide sufficient conditions for obtaining an approximate reduced system that remains valid up to second order. However, in general, these equations cannot be fulfilled simultaneously.
As detailed in the next sections, we can circumvent this problem by prioritizing  Eq.~\eqref{comp eq W app} when solving the compatibility equations. This means that we find a matrix
$\bm{\lambda}=(\lambda_{\nu \rho})_{\nu,\rho}$
and vectors $\bm{\widehat{a}}_1, \cdots, \bm{\widehat{a}}_n$
that approximately fulfill Eq.~\eqref{comp eq W app} but not necessarily Eq.~\eqref{comp eq K app}. 
Then, for all $\nu, \rho \in \{1, \cdots, n\}$, we set $\mu_{\nu \rho}$ to be the parameter that minimizes the quadratic error of Eq.~\eqref{comp eq K app}, that is,
\begin{equation}
    \mu_{\nu \rho} =  { \bm{\widehat{a}}_\nu^+ \bm{K}_{\nu \rho} \bm{\widehat{a}}_\nu}= \frac{ \bm{\widehat{a}}_\nu^T \bm{K}_{\nu \rho} \bm{\widehat{a}}_\nu} { \| \bm{\widehat{a}}_\nu \|^2}
\end{equation}
(see Lemma \ref{lemma min quadratic error}).
We can thus reasonably assume that Eqs.~\eqref{comp eq W app 4} and~\eqref{lambda nu rho 2} hold. We also assume that Eqs.~\eqref{comp eq K app 4} hold ---otherwise the observables' dynamics cannot be expressed in a closed form--- but not necessarily Eqs.~\eqref{mu nu rho 2}, so that the approximate reduced dynamics has an additional correction term:
\begin{equation}
\begin{array}{lll}
\cal{\dot{X}}_\nu 
&\approx&
\displaystyle f(\cal{X}_\nu)  + \sum \limits_{\rho = 1}^n 
\cal{W}_{\nu \rho} \, g(\cal{X}_\nu, \cal{X}_{\rho})
+ \sum \limits_{\rho = 1}^n 
g_1(\cal{X}_\nu, \cal{X}_{\rho}) \left( \mu_{\nu \rho} - \cal{W}_{\nu \rho} \right) \cal{X}_\nu
\end{array}
\label{approx dynamics + corr mu app}
\end{equation}
with $\{ \cal{W}_{\nu \rho} \}_{\nu,\rho}$ as defined in Proposition \ref{prop spectral} by Eq.~(\ref{def calW nu rho app}). Whenever Eqs.~(\ref{mu nu rho 2}) hold, ${\cal W}_{\nu \rho} = \mu_{\nu \rho}$ for all $\nu,\rho$ and we recover the reduced dynamics of Proposition \ref{prop spectral}.

\section{Equivalent forms for the compatibility equations when the adjacency matrix is positive} \label{sec: app equivalence comp eqs}

Here we show that, whenever the adjacency matrix $\bm{W}$ is positive and we want the vectors $\bm{\widehat{a}}_1, \cdots, \bm{\widehat{a}}_n$ to be also positive, we can transform the compatibility equations into another set of equivalent, decoupled equations.
We start by presenting some useful propositions.

\begin{proposition} \label{prop 1}

\emph{Let $\bm{A}$, $\bm{B}$ be two positive matrices of dimension $n \times m$ and $m \times n$, respectively. Let $\bm{u}$, $\bm{v}$ be two non-zero vectors of dimension $n$ and $m$, respectively. Then:}
\begin{enumerate}

\item \emph{There exists a scalar $\lambda > 0$ that is a dominant eigenvalue of both matrices $\bm{AB}$ and $\bm{BA}$ (that is, any other eigenvalue $\lambda'$ of $\bm{AB}$ or $\bm{BA}$ is smaller in modulus: $|\lambda'| < \lambda$). The multiplicity of $\lambda$ as an eigenvalue of both $\bm{AB}$ and $\bm{BA}$ is 1.}

\item \emph{If $\bm{u}$ and $\bm{v}$ are, respectively, eigenvectors of $\bm{AB}$ and $\bm{BA}$ associated to the dominant eigenvalue $\lambda$, then there exist scalars $\alpha, \beta \neq 0$ such that }
\begin{equation}
\begin{array}{lll}
\bm{A v} &=& \alpha \bm{u} \\
\bm{B u} &=& \beta \bm{v}.
\end{array}
\end{equation}
\end{enumerate}

\end{proposition}

\begin{proof}\hfill

\begin{enumerate}

\item $\bm{AB}$ and $\bm{BA}$ are positive matrices. Then, the Perron-Frobenius Theorem states that they have dominant eigenvalues $\lambda$ and $\lambda'$, respectively, that are positive and have multiplicity 1. We only need to prove that $\lambda = \lambda'$. Let $\bm{u}$ be an eigenvector of $\bm{AB}$ associated to $\lambda$:
$$
\bm{ABu} = \lambda \bm{u}.
$$
Left-multiplying this equation by $\bm{B}$ we have
$$
\bm{BA(Bu)} = \lambda \bm{Bu}.
$$
$\bm{Bu} \neq \bm{0}$ because otherwise we would have $\bm{0} = \bm{ABu} = \lambda \bm{u}$ and this is absurd since $\lambda$ and $\bm{u}$ are non-zero. Therefore, $\bm{Bu}$ is an eigenvector of $\bm{BA}$ with eigenvalue $\lambda$, which implies $|\lambda| \leq \lambda'$. Inverting the roles of $\bm{A}$ and $\bm{B}$ we obtain $\lambda' \leq \lambda$. We conclude that $\lambda = \lambda'$.

\item Let $\bm{u}$ and $\bm{v}$ be, respectively, eigenvectors of $\bm{AB}$ and $\bm{BA}$ associated to the dominant eigenvalue $\lambda$. In the proof of point 1. we have shown that $\bm{Bu}$ and $\bm{Av}$ are eigenvectors of $\bm{B A}$ and $\bm{A B}$, respectively, with eigenvalue $\lambda$. We also know that $\lambda$ has multiplicity 1 in both cases, which means that the eigenspaces associated to $\lambda$ for $\bm{AB}$ and $\bm{BA}$ have dimension 1. Therefore, $\bm{Bu}$ and $\bm{Av}$ have to be multiples of $\bm{v}$ and $\bm{u}$, respectively: there exist scalars $\alpha, \beta \neq 0$ such that
$$
\begin{array}{lll}
\bm{A v} &=& \alpha \bm{u} \\
\bm{B u} &=& \beta \bm{v}.
\end{array}
$$
\end{enumerate} 

\hfill\end{proof}

\begin{proposition}  \label{prop 2}

\emph{Let $\bm{A}$, $\bm{B}$ be two positive matrices of dimension $n \times m$ and $m \times n$, respectively. Let $\bm{u}$, $\bm{v}$ be two vectors of dimension $n$ and $m$, respectively, such that all their entries are positive. Then, the following statements are equivalent:}
\begin{enumerate}

\item \emph{There exist scalars $\alpha, \beta > 0$ such that }
\begin{equation}
\begin{array}{lll}
\bm{A v} &=& \alpha \bm{u} \\
\bm{B u} &=& \beta \bm{v}.
\end{array}
\label{result 2.2}
\end{equation}

\item \emph{ There exists a scalar $\lambda > 0$ such that }
\begin{equation}
\begin{array}{lll}
\bm{A B u} &=& \lambda \bm{u} \\
\bm{B A v} &=& \lambda \bm{v}.
\end{array}
\end{equation}

\end{enumerate}

\end{proposition}

\begin{proof}

We start by showing that 1. implies 2. If we left-multiply the first equality in Eq.~\eqref{result 2.2} by $\bm{B}$ and then use the second one, we get 
\begin{equation}
\begin{array}{lllll}
\bm{B A v} &=& \alpha \bm{B u} &=& \alpha \beta \bm{v}.
\end{array}
\end{equation}
Analogously,
\begin{equation}
\begin{array}{lllll}
\bm{A B u} &=& \beta \bm{A v} &=& \alpha \beta \bm{u}.
\end{array}
\end{equation}
Therefore, 2. holds with $\lambda = \alpha \beta > 0$.

Let us show now that 2. implies 1. From 2. we have that $\bm{u}$ and $\bm{v}$ are, respectively, eigenvectors of $\bm{AB}$ and $\bm{BA}$ with eigenvalue $\lambda$. Since all the components of $\bm{u}$ and $\bm{v}$ are positive and $\bm{AB}$, $\bm{BA}$ are positive matrices, by the Perron-Frobenius Theorem they must be dominant eigenvectors of $\bm{AB}$ and $\bm{BA}$. This means that $\lambda$ is the dominant eigenvalue of both matrices. Thus, according to Proposition~\ref{prop 1}, there exist scalars $\alpha, \beta \neq 0$ such that Eq.~\eqref{result 2.2} holds. Moreover, $\alpha, \beta$ are positive because  $\bm{AB}, \bm{BA}$ are positive matrices and $\bm{u,v}$ have positive entries.\hfill\end{proof}

Notice that going from 1. to 2. is straightforward and does not require the positiveness hypothesis on $\bm{A}$, $\bm{B}$, $\bm{u}$ and $\bm{v}$. These conditions are nonetheless needed to deduce 1. from 2.
The following corollary is a direct application of Proposition \ref{prop 2} to the compatibility equations.

\begin{corollary} \label{corollary 1}

For $\nu, \rho \in \{1, \cdots, n\}$, let $\bm{W}_{\nu \rho}$ be a positive matrix of dimension $m_\nu \times m_\rho$ and let $\bm{\widehat{a}}_\nu$ be a positive vector of dimension $m_\nu$. Then, the following statements are equivalent:
\begin{enumerate}

\item There exists a set of positive scalars $\{ \lambda_{\nu \rho} \}_{\nu,\rho}$ such that
\begin{equation}
\begin{array}{lllll}
\bm{W}_{\nu \nu}^T \bm{\widehat{a}}_\nu &=& \lambda_{\nu \nu} \bm{\widehat{a}}_\nu, & & \nu \in \{1, \cdots, n\} \\
\bm{W}_{\nu \rho}^T \bm{\widehat{a}}_\nu &=& \lambda_{\nu \rho} \bm{\widehat{a}}_\rho, & & \nu, \rho \in \{1, \cdots, n\}, \, \nu \neq \rho.
\end{array}
\label{comp eqs 1 app}
\end{equation}

\item There exists a set of positive scalars $\{ \lambda'_{\nu \rho} \}_{\nu,\rho}$ such that 
\begin{equation}
\begin{array}{lllll}
\bm{W}_{\nu \nu}^T \bm{\widehat{a}}_\nu &=& \lambda'_{\nu \nu} \bm{\widehat{a}}_\nu, & & \nu \in \{1, \cdots, n\} \\
\bm{W}_{\rho \nu}^T \bm{W}_{\nu \rho}^T \bm{\widehat{a}}_\nu &=& \lambda_{\nu \rho}' \bm{\widehat{a}}_\nu, & & \nu, \rho \in \{1, \cdots, n\}, \, \nu \neq \rho.
\end{array}
\label{comp eqs 2 app}
\end{equation}

\end{enumerate}

Moreover, the relation between the scalars of statements 1. and 2. is
\begin{equation}
\begin{aligned}
\lambda_{\nu \nu}' &= \lambda_{\nu \nu}  & \\
\lambda_{\nu \rho}' &= \lambda_{\nu \rho} \lambda_{\rho \nu} &\, \text{ for } \nu \neq \rho.
\end{aligned}
\end{equation}
This allows us to transform the original compatibility equations, Eqs.~ \eqref{comp eq K app}--\eqref{comp eq W app}, into the \emph{decoupled compatibility equations}
\begin{subequations}
\begin{align}
    \bm{K}_{\nu \rho} \bm{\widehat{a}}_\nu = \mu_{\nu \rho} \bm{\widehat{a}}_\nu
    \label{comp eq K decoupled app} \\
    \bm{W'}_{{\nu \rho}} \bm{\widehat{a}}_\nu = \lambda_{\nu \rho}' \bm{\widehat{a}}_\nu,
    \label{comp eq W decoupled app}
\end{align}
\end{subequations}
where
\begin{equation}
\begin{array}{lll}
\lambda_{\nu \rho}' :=\left\lbrace
\begin{array}{lll}
\lambda_{\nu \nu} && \text{ if } \nu = \rho \\
\lambda_{\nu \rho} \lambda_{\rho \nu} && \text{ if } \nu \neq \rho\\
\end{array}
\right. ,

&&

\bm{W'}_{\nu \rho} := \left\lbrace
\begin{array}{lll}
\bm{W}_{\nu \nu}^T && \text{ if } \nu = \rho \\
\bm{W}_{\rho \nu}^T \bm{W}_{\nu \rho}^T && \text{ if } \nu \neq \rho \\
\end{array}
\right. .
\end{array}
\end{equation}

\end{corollary}

Once Eqs.~\eqref{comp eq K decoupled app}--\eqref{comp eq W decoupled app} are solved and $\bm{\widehat{a}}_\nu$ is known for all $\nu$, the original set of compatibility equations \eqref{comp eq K app}--\eqref{comp eq W app} is fulfilled and the set of scalars $\{ \mu_{\nu \rho}, \lambda_{\nu \rho} \}_{\nu,\rho}$ can be determined via Eqs.~\eqref{mu nu rho 2}--\eqref{lambda nu rho 2} or \eqref{mu nu rho 1}--\eqref{lambda nu rho 1}.

As we have noticed after the proof of the second result, the positiveness of $\bm{W}_{\rho \nu}^T$ and $\bm{\widehat{a}}_\nu$ for all $\nu, \rho$ is needed to obtain the equivalence between Eqs.~\eqref{comp eqs 1 app} and Eqs.~\eqref{comp eqs 2 app}. If we relax this hypothesis, solving the decoupled equations might not be sufficient for solving the original set of equations. 

The decoupled compatibility equations \eqref{comp eq K decoupled app}--\eqref{comp eq W decoupled app} are not simultaneously solvable in general. In the next sections we present a possible strategy to find an approximate solution.

\section{An approximate solution to the compatibility equations that involve the adjacency matrix} \label{sec: app solution comp eqs W}

Let us focus on approximately solving Eqs.~\eqref{comp eq W decoupled app} for a fixed $\nu$ and variable $\rho$. For this, we assume that the scalars $ \lambda'_{\nu 1}, \cdots, \lambda'_{\nu n}$ are the dominant eigenvalues of the matrices involved (this would be the case if Eqs.~\eqref{comp eq W decoupled app} could be solved \emph{exactly}). Our goal then is to find a vector that has sum 1 and minimizes the sum of the corresponding quadratic errors. For now we will relax the condition of the vector having positive entries. We can formulate the problem as follows:

\begin{problem} \label{problem}

\emph{Given a set of $m \times m$ positive matrices $\{\bm{M}_i\}_{i=1}^n$ and scalars $\{\lambda_i\}_{i=1}^n$, find a vector $\bm{a}=(a_i)_{i=1}^m \in \mathbb{R}^m$ with $\sum \limits_{i=1}^m a_i = 1$ and such that the following error is minimal:}
\begin{equation}
E( \bm{a} ) := \Vert \bm{M}_1 \bm{a} - \lambda_1 \bm{a} \Vert^2 + \cdots +  \Vert \bm{M}_n \bm{a} - \lambda_n \bm{a} \Vert^2.
\end{equation}

\end{problem}

\begin{proposition} \label{prop 3}

Let $\bm{u}_1, \cdots, \bm{u}_r$ be $r$ linearly independent, non-negative vectors in $\mathbb{R}^m$. Without loss of generality we can assume that these vectors are normalized: $\sum \limits_{i=1}^m [\bm{u}_s]_i = 1$ for all $s \in \{1, \cdots, r\}$. Then, a solution $\bm{a} \in \mathbb{R}^m$ to Problem \ref{problem} 
of the form
\begin{equation}
\bm{a} ( \bm{x} ) := x_1 \bm{u}_1 + \cdots + x_r \bm{u}_r
\label{a(x) app}
\end{equation}
must satisfy the following condition: the vector $\bm{y} := (x_1, \cdots, x_r, K )^T$, for some non-zero constant $K$, is a solution to the system of $r+1$ linear equations
\begin{equation}
\begin{array}{lll}
\bm{\hat{C} y} &=&  ( 0, \cdots, 0, 1 )^T
\end{array}
\label{eq Cy}
\end{equation}
where
$
\bm{\hat{C}} = \left(
\begin{array}{c | c} 
\bm{C} & - \bm{1} \\
\hline
\bm{1}^T & 0
\end{array}
\right)
$
and 
$\bm{C} = (c_{s t})_{s,t}$ is the $r \times r$ matrix defined by
\begin{equation}
c_{st} := \sum \limits_{j=1}^n \langle \bm{M}_j \bm{u}_s - \lambda_j \bm{u}_s, \bm{M}_j \bm{u}_t - \lambda_j \bm{u}_t \rangle.
\end{equation}

\end{proposition}

\begin{proof}

We are looking for a solution $\bm{a}$ within the subspace spanned by $\bm{u}_1, \cdots, \bm{u}_r$. The vector $\bm{a}$ can thus be expressed as
\begin{equation}
\bm{a}( \bm{x} ) = x_1 \bm{u}_1 + \cdots + x_r \bm{u}_r
\end{equation}
for some $\bm{x} = (x_1, \cdots, x_r)^T$. Since  $\sum \limits_{i=1}^m [\bm{u}_s]_i = 1$ by assumption, the condition $\sum \limits_{i=1}^m a_i = 1$ requires $\sum \limits_{i=1}^r x_i = 1$.
We thus have to find $\bm{x}$ such that $\bm{a}(\bm{x})$ minimizes $E( \bm{a}(\bm{x}) )$ subject to the constraint
\begin{equation}
\sum \limits_{i=1}^r x_i = 1. 
\end{equation}

This is a minimization problem with a single constraint that can be solved using the method of the Lagrange multipliers: considering the Lagrangian function
\begin{equation}
\cal{L} ( \bm{x}, K ) := \sum \limits_{j=1}^n \Vert \bm{M}_j \bm{a}(\bm{x}) - \lambda_j \bm{a}(\bm{x}) \Vert^2
+ 2 K \left(1 - \sum \limits_{s=1}^r x_s \right)
= E( \bm{a}(\bm{x}) ) + 2 K \left(1 - \sum \limits_{s=1}^r x_s \right),
\end{equation}
a local solution to the problem necessarily fulfills
\begin{equation}
\begin{array}{lll}
\displaystyle \frac{\partial}{\partial x_i} \cal{L}( \bm{x}, K ) = 0 \text{\hspace{0,3cm}} \forall i, &&
\displaystyle \frac{\partial}{\partial K} \cal{L}( \bm{x}, K ) = 0. \\
\end{array}
\label{Lagrange equations}
\end{equation}
Let us notice that, for a given $j \in \{1, \cdots, n \}$,
\begin{equation}
\begin{array}{lll}
\Vert \bm{M}_j \bm{a} (\bm{x}) - \lambda_j \bm{a} (\bm{x}) \Vert^2 
&=& \displaystyle \Vert \sum \limits_{s = 1}^r x_s \left( \bm{M}_j \bm{u}_s - \lambda_j \bm{u}_s \right) \Vert^2 \\
&=& \displaystyle \langle \sum \limits_{s = 1}^r x_s \left( \bm{M}_j \bm{u}_s - \lambda_j \bm{u}_s \right),
 \sum \limits_{t=1}^r x_t \left( \bm{M}_j \bm{u}_t - \lambda_j \bm{u}_t \right) \rangle \\

&=& \displaystyle
\sum \limits_{s,t=1}^r x_s x_t \,
\langle \bm{M}_j \bm{u}_s - \lambda_j \bm{u}_s, \bm{ M}_j \bm{u}_t - \lambda_j \bm{u}_t \rangle \\
&=& \displaystyle
\sum \limits_{s,t=1}^r x_s x_t \, c_{st}^j
 \\
 &=& \langle \bm{x}, \bm{C}_j \bm{x} \rangle,
\end{array}
\end{equation}
where $\langle \cdot, \cdot \rangle$ denotes the scalar product and we have defined
\begin{equation}
\begin{array}{lll}
\bm{C}_j := (c_{st}^j)_{s,t}, &&
c_{st}^j := \langle \bm{M}_j \bm{u}_s - \lambda_j \bm{u}_s, \bm{M}_j \bm{u}_t - \lambda_j \bm{u}_t \rangle.
\end{array}
\end{equation}
Due to the symmetry of the scalar product, $\bm{C}_j$ is a symmetric matrix for all $j$. Introducing the matrix $\bm{C} :=  \sum \limits_{j=1}^n \bm{C}_j$ we have
\begin{equation}
\begin{array}{lll}
\cal{L}( \bm{x}, K ) &=& \displaystyle \sum \limits_{j=1}^n \langle \bm{x}, \bm{C}_j \bm{x} \rangle
+ 2 K \left( 1 - \sum_{s=1}^r x_s \right) \\
&=&\displaystyle \langle \bm{x}, \bm{C x} \rangle
+ 2 K \left( 1 - \sum_{s=1}^r x_s \right)
\end{array}
\end{equation}
and
\begin{equation}
\begin{array}{lll}
\displaystyle \frac{\partial}{\partial x_i} \cal{L}( \bm{x}, K ) 
&=& \displaystyle \langle \frac{\partial}{\partial x_i} \bm{x}, \bm{C x} \rangle
+ \langle \bm{x}, \bm{C} \frac{\partial}{\partial x_i} \bm{x} \rangle - 2 K \\
&=& \displaystyle 
2 \langle \frac{\partial}{\partial x_i} \bm{x}, \bm{C x} \rangle - 2 K \\
&=& \displaystyle 2 \left( [\bm{C x}]_i - K \right),
\end{array}
\end{equation}
where in the second equality we have taken into account that $\bm{C}$ is symmetric.
We can finally express the set of equations~\eqref{Lagrange equations} as
\begin{equation}
\begin{array}{lll}
\bm{C x} &=& \displaystyle K \bm{1} \\
\displaystyle \sum \limits_{s=1}^r x_s &=& 1, \\
\end{array}
\label{system alpha}
\end{equation}
where $\bm{1} = (1, \cdots, 1)^T$.
This is a system of $r+1$ linear equations that can in turn be rewritten as
\begin{equation}
\begin{array}{lll}
\bm{\hat{C} y} &=&  ( 0, \cdots, 0, 1 )^T
\end{array}
\end{equation}
with
$\bm{\hat{C}} := \left(
\begin{array}{c | c} 
\bm{C} & - \bm{1} \\
\hline
\bm{1}^T & 0
\end{array}
\right)
$
and $\bm{y} := (x_1, \cdots, x_r, K )$.\hfill\end{proof}

We observe the following:

\begin{itemize}

\item \emph{The error associated to a solution $(\bm{x}, K)$ is given by $K$.}
If $(\bm{x}, K)$ is a solution to Eq.~\eqref{system alpha}, then the error $E( \bm{a} (\bm{x}) )$ and the Lagrangian function $\cal{L}(\bm{x}, K)$ take the same value and this is
\begin{equation}
\begin{array}{lll}
E( \bm{a} (\bm{x}) ) 
&=& \langle \bm{x}, \bm{C x} \rangle \\ 
&=& \langle \bm{x}, K \bm{1} \rangle \\ 
&=& K. \\ 
\end{array}
\end{equation}

\item \emph{All the solutions to Eq.~\eqref{system alpha} produce the same error.} Suppose, on the contrary, that there are two solutions, $s_1 = (\bm{x}_1, K_1), s_2 = (\bm{x}_2, K_2)$ that produce different errors, that is, with $K_1 \neq K_2$. Any point in the segment that connects $s_1$ and $s_2$, $s(h) = (1-h) s_1 + h s_2$, $h \in (0,1)$, is also a solution and its corresponding Lagrangian function is $\cal{L}(s(h)) = (1-h) K_1 + h K_2$. Since $K_1 \neq K_2$, this means that the directional derivative of the Lagrangian function along this segment is non-zero for all $h$, which contradicts the fact that the solutions to Eq.~\eqref{system alpha} are the points for which all the directional derivatives of the Lagrangian function are zero. We conclude that $K_1 = K_2$.

\end{itemize}

We can apply this result to our set of decoupled compatibility equations that involve the adjacency matrix. If we do not impose that $\bm{\widehat{a}}_\nu$ has positive entries and we assume that it lies in the subspace spanned by a collection of $r$ vectors $\bm{u}_1, \cdots, \bm{u}_r \in \mathbb{R}^{m_\nu}$ with $\sum \limits_{i=1}^{m_\nu} [\bm{u}_s]_i = 1$ for all $s$, then
\begin{equation}
\bm{\widehat{a}}_\nu = x_1 \bm{u}_1 + \cdots + x_r \bm{u}_r
\end{equation}
and the $x_1, \cdots, x_r$ parameters are obtained by solving Eq.~\eqref{eq Cy} with matrix $\bm{C}$ defined by
\begin{equation}
c_{st} := \sum \limits_{\rho=1}^n \langle \bm{W'}_{\nu \rho} \bm{u}_s - \lambda'_{\nu \rho} \bm{u}_s,  \bm{W'}_{\nu \rho} \bm{u}_t - \lambda'_{\nu \rho} \bm{u}_t \rangle .
\end{equation}

A solution that is not restricted to a particular subspace is obtained when $r=m_\nu$ and $\bm{u}_1, \cdots, \bm{u}_{m_\nu}$ is the canonical basis of $\mathbb{R}^{m_\nu}$. This solution is the one with the smallest error. We call it the \emph{optimal} solution.

However, it can be useful to seek a solution in a subspace of dimension $r \ll m_\nu$ because the associated system of linear equations will be of smaller dimension and, therefore, easier and faster to solve. Taking into account that the parameters $ \lambda'_{\nu 1},\cdots,\lambda'_{\nu n}$ are in fact the dominant eigenvalues of matrices $\bm{W'}_{\nu 1}, \cdots, \bm{W'}_{\nu n}$, we hypothesized that the optimal solution is close to the subspace spanned by the dominant eigenvectors of these matrices. This is clearly the case when $n=1$ and when the matrices share their dominant eigenspace.

To test this hypothesis, we randomly generated sets of $n$ matrices of dimension $m \times m$ and we compared the optimal solution with the solution that is restricted to the subspace spanned by the dominant eigenvectors, for different choices of $n$ and $m$. The results show that the error associated to the restricted solution is only slightly larger than that of the optimal solution (Fig.~\ref{fig optimal vs restricted solution}A) and that the two solutions are very similar (Fig.~\ref{fig optimal vs restricted solution}B), which suggests that the solution that is restricted to the subspace of dominant eigenvectors is a good approximation to the optimal solution. As expected, the two solutions coincide whenever $n=m$.

\begin{figure}[ht] 
\centering
\includegraphics[width=0.8\linewidth ]{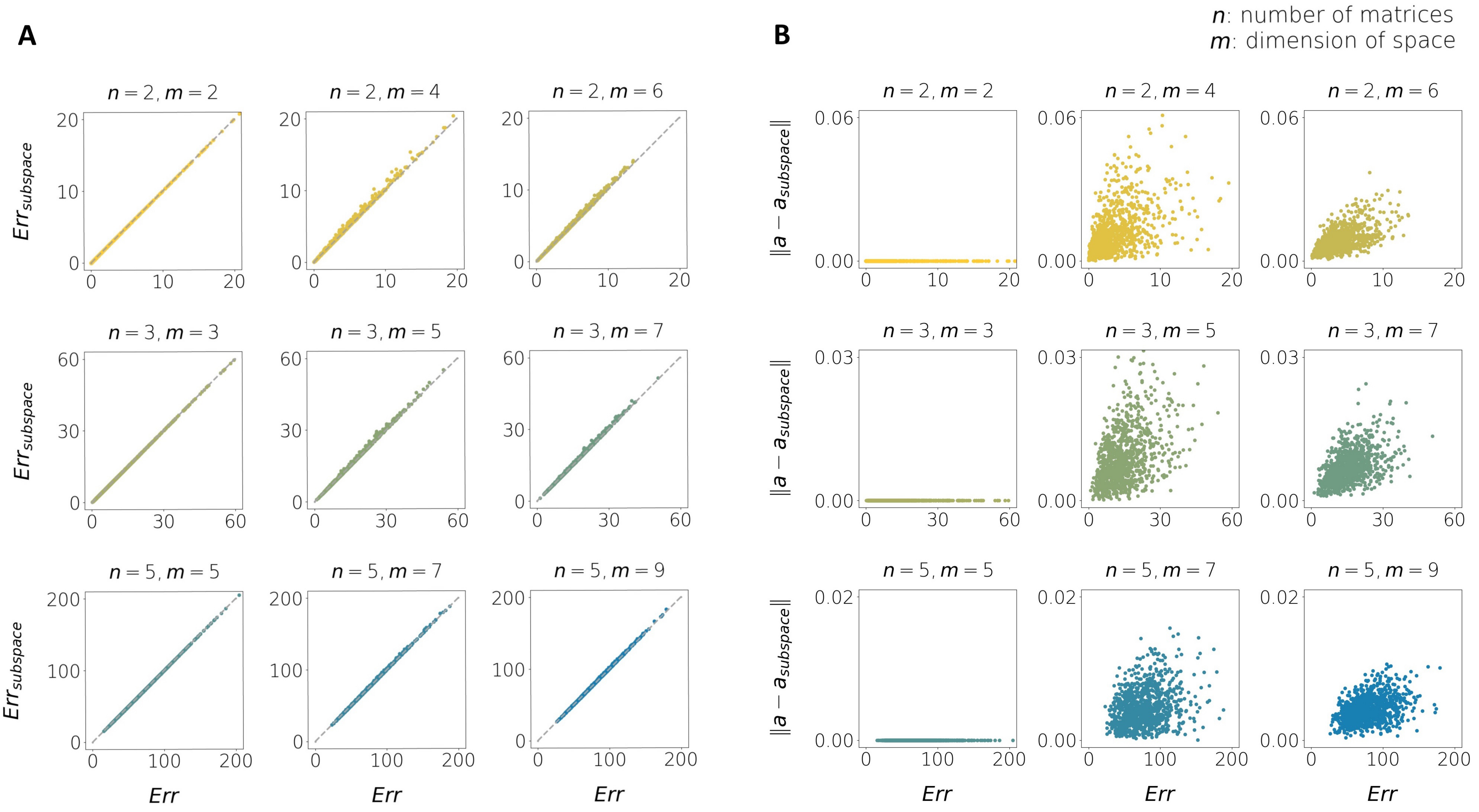}
\caption{\small Comparison of the optimal and restricted solutions to Problem \ref{problem}, defined by Eqs.~\eqref{a(x) app}),~\eqref{eq Cy}, the restricted solution being the one that lies in the subspace spanned by the dominant eigenvectors of the matrices involved. For each choice of $n$ (number of matrices) and $m$ (dimension of space), we randomly generated 1000 sets of $m \times m$ matrices, we found their dominant eigenvalues and both solutions were computed. For a fixed $n$, the elements of the $i$-th matrix were created independently within the range $(0, 1+5 i)$, $i \in \{1, \cdots, n\}$.
{\textbf A.} Error associated to the restricted solution ($Err_\text{subspace}$) versus error associated to the optimal solution ($Err$).
{\textbf B.} Euclidean distance between the two solutions ($\bm{a}$, $\bm{a}_\text{subspace}$) versus error associated to the optimal solution.
}
\label{fig optimal vs restricted solution}
\end{figure}

\section{Construction of heterogeneous networks with communities} \label{sec: app heterogeneous networks}

A heterogeneous network with block structure and in/out-degree variability is constructed as follows. First, a partition of the nodes into $n$ groups is defined. Let $m_\nu$ be the size of group $G_\nu$. For every ordered pair of group indices $(\nu,\rho)$, a parameter $p_{\nu \rho}$ defines the mean connection density of interactions from $G_\rho$ to $G_\nu$. Every node $i$ in $G_\nu$ is assigned a collection of hidden in/out-degrees from/to the other groups: 
$\bm{\kappa^{i,\textbf{in}}} = ( \kappa^{i,\text{in}}_1, \cdots, \kappa^{i,\text{in}}_n )$,
$\bm{\kappa^{i,\textbf{out}}} = ( \kappa^{i,\text{out}}_1, \cdots, \kappa^{i,\text{out}}_n )$
such that
\begin{equation}
\begin{array}{lll}
\langle \kappa^{i,\text{in}}_\rho \rangle &=& m_\rho \, p_{\nu \rho} \\
\langle \kappa^{i,\text{out}}_\rho \rangle &=& m_\rho \, p_{\rho \nu}.
\end{array}
\end{equation}
Once the hidden degrees are specified, a connection from node $j \in G_\rho$ to node $i \in G_\nu$ is created with probability
\begin{equation}
p_{ij} = \frac{ \kappa^{i,\text{in}}_\rho \, \kappa^{j,\text{out}}_\nu }{ m_\nu \, m_\rho \, p_{\nu \rho} }.
\end{equation}
The hidden degrees can follow any distribution provided that their expectation is the one specified above and that the hidden degrees of every node are independent of those of any other node. We can also incorporate a correlation between the hidden in- and out-degrees of single nodes. In our example networks they are uniformly distributed and the hidden in/out-degrees of a node from/to its own group are correlated with correlation coefficient $\rho_\text{in/out} = 0.8$.

\section{Partition refinement} \label{sec: app part refinement}

Given a network and a node partition, we refine the partition (that is, we split the existing groups into smaller subgroups) so that the weighted in/out-degree variability of nodes that are in the same subgroup is reduced. The process is as follows. We take two parameters $v_\text{in}, v_\text{out}$ that define the maximal in- and out-degree variability allowed in the new partition. This means that the new partition has to be such that all the weighted in- and out-degrees (coming from and ending at nodes in all the other groups) of two nodes that are in the same group can differ, at most, by $v_\text{in}$ and $v_\text{out}$, respectively. For this we first compute all the weighted in- and out-degrees of nodes in each group, coming and ending at all the other groups (i.e., for each node we have $n$ in-degrees and $n$ out-degrees, where $n$ is the number of groups in the original partition). Then, for each pair of groups $G_\nu$, $G_\rho$, we order the in(out)-degrees of nodes in $G_\nu$ from (to) $G_\rho$ and we classify these degrees into categories so that the difference between the minimal and maximal degree within each category is smaller than the desired threshold $v_\text{in}$ ($v_\text{out}$) and so that the number of categories is as small as possible. Now we classify all the nodes in $G_\nu$ according to these degree categories: two nodes end up in the same subgroup whenever all their degrees have fallen in the same category.

\section{Computing the bifurcation diagrams} \label{sec: app bif diagrams}

To compute the bifurcation diagram of a system, we first create a network instantiation or take a network from given data. If the network is binary, the adjacency matrix is converted into a positive matrix by setting all the missing connections to a very small value $\epsilon > 0$. We then vary the overall strength of connections by multiplying all the interaction weights by a common factor $d$. For each value of $d$ we compute the homogeneous and spectral reductions and we integrate both the original and the reduced dynamics to equilibrium. Once the equilibrium values of the $n$ observables (exact or reduced) are known, we compute their weighted average according to group size:
\begin{equation}
\langle \cal{X} \rangle := \frac{1}{N} \sum \limits_{\nu = 1}^n m_\nu \, \cal{X}_\nu.
\end{equation}
We also define $\cal{K}_\nu$ as the weighted in-degree of observable $\nu$ in the reduced system:
\begin{equation}
\cal{K}_\nu := \sum \limits_{\rho = 1}^n \cal{W}_{\nu \rho},
\end{equation}
and from it we can compute the average in-degree in the reduced system, weighted by the group size:
\begin{equation}
\langle \cal{K} \rangle := \frac{1}{N} \sum \limits_{\nu = 1}^n m_\nu \, \cal{K}_\nu.
\end{equation}
We finally generate a diagram that shows the state of the observable average $\langle \cal{X} \rangle$ (exact and reduced) at equilibrium as a function of  $\langle \cal{K} \rangle$.

\end{document}